\documentclass[a4paper,11pt]{article}
\usepackage{fullpage}
\usepackage[mac]{inputenc}
\usepackage[T1]{fontenc}
\usepackage[english]{babel}
\usepackage{amsmath,amsthm,amssymb}
\usepackage{bbm}
\usepackage{verbatim}
\usepackage{graphics}
\usepackage{psfrag}
\usepackage{color, epsfig}
\usepackage[only,llbracket,rrbracket]{stmaryrd}

\def\epsilon{\varepsilon}
\def\phi{\varphi}

\def\N{\mathbb{N}}
\def\Z{\mathbb{Z}}

\def\R{\mathbb{R}} 

\def\T{\mathbb{T}}

\def\S{\mathbb{S}}

\def\II{\parbox[][0.6cm][c]{0cm}{\ }}
\newtheorem{thm}{Theorem}[section]

\newtheorem{prop}{Proposition}[section]
\newtheorem{lem}{Lemma}[section]

\numberwithin{equation}{section}
\theoremstyle{remark}
{\vskip 0.5cm}

\newtheorem{rque}{\textbf{Remark}}[section]{\vskip 0.5cm}
\newtheorem*{ack}{\textbf{Acknowledgements}}{\vskip 0.5cm} 
\title{On the controllability of the Vlasov-Poisson system in the presence of external force fields}
\author{Olivier Glass\footnote{Ceremade, Université Paris-Dauphine (glass@ceremade.dauphine.fr)} \quad \quad \quad \quad Daniel Han-Kwan\footnote{Département de Mathématiques et Applications, Ecole Normale Supérieure (hankwan@dma.ens.fr)} }
\begin{document}
 \maketitle
 \begin{abstract}
In this work, we are interested in the controllability of Vlasov-Poisson systems in the presence of an external force field (namely a bounded force field or a magnetic field), by means of a local interior control. We are able to extend the results of \cite{OG03}, where the only present force was the self-consistent electric field.
 \end{abstract}

\section{Introduction and main results}
We consider the controllability of the Vlasov-Poisson system in the periodic domain $\T^n$ (where $n$ is the space dimension), which describes the evolution of a population of electrons in a neutralizing background of fixed ions, under the influence of a self-generated electric field. The control questions are addressed by means of an interior control located in an open set $\omega$ of the domain, which is a priori arbitrary. We assume in this paper that the charged particles evolve with the influence of an additional \emph{fixed} external force, denoted by $F(t,x,v)$ (at least with Lipschitz regularity and a sublinear growth at infinity in velocity). The equations read:
  \begin{equation}
 \label{Vlasov}
 \partial_t f +v.\nabla_x f +  F(t,x,v).\nabla_v f + \nabla_x \Phi .\nabla_v f=\mathbbm{1}_\omega G, \quad x\in \mathbb{T}^n, \quad v \in \mathbb{R}^n
 \end{equation} 
  \begin{equation}
  \label{Poisson}
 \Delta_x \Phi = \int_{\mathbb{R}^n} f dv -\int_{\T^{n} \times \mathbb{R}^n} f dv dx,
 \end{equation} 
  \begin{equation}
  \label{CondInitiale}
 f_{\vert t=0}= f_0.
 \end{equation}
In these equations, $f(t,x,v)$ is the so-called distribution function, which describes the density of particles at time $t\in \R^+$, at position $x \in \T^n$ and velocity $v \in \R^n$. The initial density distribution $f_0(x,v)$ is a non-negative integrable function. The right-hand side of the transport equation $\mathbbm{1}_\omega G$ is a source term describing emission and absorption of particles, supported in $\omega$. Moreover, to preserve global neutrality, $G$ has to satisfy the following constraint:
\[
\forall t \in \R^+, \quad \int_{\T^n \times \R^n} \mathbbm{1}_\omega G \, dv \, dx =0.
\]
We normalize here the torus so that its Lebesgue measure is $1$. 
\par
The controllability problem is the following. Let $f_1(x,v)$ be another non-negative integrable function satisfying $f_1 \geq 0$ and 
\[
\int f_1 dv dx =\int f_0 dv dx,
\] 
and let $T>0$ a fixed time. The question is: is it possible to find a control $G$ such that:
 \begin{equation}
 \label{CondFinale}
 f(T,x,v)= f_1(x,v).
 \end{equation}
When the only acting force is the self-consistent electric field (that is when $F=0$), the first author provided in \cite{OG03} some positive answers to the question. More specifically, two kinds of results were obtained: first local controllability (which means that $f_0$ and $f_1$ are small in some weighted $L^\infty$ norm) were obtained in two dimensions, for an arbitrary control zone $\omega$. Global controllability results (without restriction on the size of $f_0$ and $f_1$)  in any dimension was also obtained, provided that the control zone $\omega$ contains the
image of a hyperplane of $\R^n$ by the canonical surjection (which is called a hyperplane of the torus in \cite{OG03}).
 The proofs of these results relied on the nice geometry of free transport in the torus: we shall recall their principle in a subsequent paragraph. \par
When one considers a non-trivial external force $F $, the underlying dynamical system is more complicated; thus the characteristics can have a complex geometry, making the generalization not straightforward from the case $F=0$. \par
\ \par
In this paper, we are able to extend results of \cite{OG03} for the two following classes of force fields:
\begin{itemize}
\item {The case of bounded force fields} $F \in L^\infty_{t}W^{1,\infty}_{x,v}$.
\item {In two dimensions, the case of Lorentz forces for magnetic fields with a fixed direction} $F(x,v)= b(x) (v_2, -v_1)$ with $b$ satisfying a certain geometric condition (which will be precisely described later).
\end{itemize}

As we will see later on, the treatment of these two cases are rather different (in particular for what concerns high velocities) and involve different strategies. As a matter of fact, we were not able to find a general strategy which would allow to treat all forces $F$ which are Lipschitz with a sublinear growth at infinity in velocity.

\vspace{10pt}

Let us now briefly review the existing results on the Cauchy theory for the Vlasov-Poisson equation posed in the whole space $\R^n$ or in the torus $\T^n$.
In this work, we will only focus on strong solutions (at least with a $\mathcal{C}^1$ regularity in all variables); in the case where $F=0$, the first results for such solutions are due Ukai and Okabe \cite{UO} who have proved global in time existence in two dimensions and  local in time existence in three dimensions, in the whole space setting. One can readily check that the proof is the same for the torus case. In three dimensions, in the whole space, global in time results were proved independently by Pfaffelmoser \cite{Pfa} and Lions and Perthame \cite{LP}. The results of Pfaffelmoser were adapted to the torus case by Batt and Rein \cite{BR}. Concerning global weak solutions, the main result is due to Arsenev \cite{Ar}. One can observe that all these results can be easily adapted to incorporate an additional external force $F$ (with $F$ satisfying the previous regularity assumptions).

We will only rely on the construction due to Ukai and Okabe in the following. We are now in position to precisely state the main results proved in this paper.

\subsection{Results in the bounded external field case}

We first consider the case where $F \in L^\infty_{t}W^{1,\infty}_{x,v}$. In this case, we are able to exactly extend those for $F=0$, that are a local and a global controllability results. The local result concerns only the dimension $n=2$, but is valid for any control zone $\omega$. On the contrary, the global result is valid for any $n$, but requires a stronger geometric assumption on the control zone $\omega$.

\begin{thm}[Local result]
\label{Theo:Bounded}
Let $n=2$. Let $F(t,x,v) \in L^\infty_{t}W^{1,\infty}_{x,v}$. Let  $\gamma >2$ and $T>0$. There exist $\kappa, \kappa' > 0$ small enough such that the following holds. Let
$f_{0}$ and $f_{1}$ be two functions in $C^{1}(\T^{2} \times \R^{2})
\cap W^{1,\infty}(\T^{2} \times \R^{2})$, satisfying the condition
that for any $(x,v) \in \T^{2} \times \R^{2}$ and $i \in \{ 0 , 1
\}$,
\begin{equation}
\label{DecroissanceInfini}
\left\{ \begin{array}{l}
{\II |f_{i}(x,v) | \leq \kappa (1 + |v|)^{-\gamma-1},} \\
{\II | \nabla_{x} f_{i} | + | \nabla_{v} f_{i} | \leq \kappa' (1 +
|v|)^{-\gamma},}
\end{array} \right.
\end{equation}
and
\begin{equation}
\label{MemeDensiteNeutralisante}
\int_{\T^{n} \times \R^{n}} f_0 = \int_{\T^{n} \times \R^{n}} f_1.
\end{equation}
%
%
Then there exists a control $G \in C^{0}([0,T] \times \T^{2} \times \R^{2})$, such that the solution of
(\ref{Vlasov})-(\ref{Poisson}) and (\ref{CondInitiale}) exists, is
unique, and satisfies (\ref{CondFinale}).
\end{thm}
\begin{thm}[Global result]
\label{Theo:BoundedGlobal}
Let $\gamma >n$ and $\kappa, \kappa' > 0$. Suppose that the regular
open set $\omega$ contains the image of a hyperplane in $\R^n$ by the
canonical surjection, supposed to be closed. Let $f_{0}$ and $f_{1}$ be two functions in
$C^{1}(\T^{n} \times \R^{n})$, satisfying the conditions
\begin{equation}
\label{DecroissanceInfini2}
\left\{ \begin{array}{l}
{\II |f_{i}(x,v) | \leq \kappa (1 + |v|)^{-\gamma-2},} \\
{\II | \nabla_{x} f_{i} | + | \nabla_{v} f_{i} | \leq \kappa' (1 +
|v|)^{-\gamma},}
\end{array} \right.
\end{equation}
and (\ref{MemeDensiteNeutralisante}).
Then there exists a control $G \in C^{0}([0,T] \times \T^{n} \times \R^{n})$,
such that the solution of
(\ref{Vlasov})-(\ref{Poisson}) and (\ref{CondInitiale}) exists, is
unique, and satisfies (\ref{CondFinale}).
\end{thm} \par
\subsection{Results in the magnetic field case}

Let us now state our result when $F$ represents an external magnetic field.
For all results dealing with this case, we will systematically assume that the space dimension $n=2$. First, let us explain the physical meaning of the system under consideration. In the physical space $\R^3$, let $(e_1, e_2, e_3)$ a fixed orthonormal base. We consider the stationary magnetic field $B$, with fixed direction $e_3$:
\[
B(x)= b(x) e_3,
\] 
where $b$ is a Lipschitz function on $\T^3$. Since $B$ has to satisfy the divergence free condition, this implies that $b$ only depends on $x_1$ and $x_2$. The associated Lorentz force writes:
\[
F= v \wedge B(x)= b(x) v^\perp,
\]
denoting $v^\perp = (v_2,-v_1, 0)$. We then restrict to distribution functions which do not depend on $x_3$ and $v_3$, so that we can restrict the study of the dynamics to the bidimensional plane $(e_1,e_2)$.
For the sake of readability, we rewrite the Vlasov-Poisson system that we study:

 \begin{equation}
 \label{VlasovCM}
 \partial_t f +v.\nabla_x f +  b(x)v^\perp.\nabla_v f + \nabla_x \Phi .\nabla_v f=\mathbbm{1}_\omega G, \quad x\in \mathbb{T}^2, \quad v \in \mathbb{R}^2
 \end{equation} 
  \begin{equation}
  \label{PoissonCM}
 \Delta_x \Phi = \int_{\mathbb{R}^2} f dv -\int_{\T^2 \times \mathbb{R}^2} f dv dx,
 \end{equation} 
  \begin{equation}
  \label{CondInitialeCM}
 f_{\vert t=0}= f_0.
 \end{equation}
 
 We now precisely state the geometric assumption we have to make on $b$.
\begin{itemize}
\item {\bf Fixed sign.} We  assume that $b$ has a fixed (say non-negative) sign. 

\item {\bf Geometric control condition.} We assume that there exists $K$ a compact set of $\mathbb{T}^2$ on which $b>0$ and which satisfies the geometric control condition:
\begin{multline} \label{GeometricCondition}
\text{For any } x \in \mathbb{T}^2 \text{ and any direction } e \in \mathbb{S}^{1}, \\
\text{ there exists } y \in \R^{+} \text{ such that } x+ye \in K.
\end{multline}

\end{itemize}

One can notice that the geometric control condition corresponds to the geometric control condition of Bardos, Lebeau and Rauch \cite{BLR} for the controllability of the wave equation. Let us underline however that here this condition concerns the magnetic field only, and not the control zone $\omega$. As we will see, this condition assures that the particles are sufficiently influenced by the magnetic field. \par
\ \par
\noindent
{\bf Examples.} Let us give some examples, where this geometric assumption is satisfied.
\begin{enumerate}
\item The most simple example that one can have in mind is the case where $b$ is positive on $\T^2$. Then taking $K= \T^2$, the geometric assumption is satisfied. Obviously, this includes the case where $b$ is a positive constant.
\item Assume that $b$ is non-negative and has finite number $N$ of zeros $x_1, ..., x_N \in \T^2$. Then there is $r$ small enough such that $K=\T^2 \backslash \cup_{i=1}^N B(x_i,r)$ is appropriate. One could also extend this consideration to the case where the zeros of $b$ are given by a sequence $(x_i)_{i \in \N}$ with a finite number of cluster points.

\item We can consider some $b$ which is identically equal to $0$ in a large set of the torus, provided the existence of some $K$ satisfying the geometric control condition. For instance, if we identify $\T^{2}$ with $[0,1]^{2}$ with periodic conditions, a subset $K$ containing $(\{ 0\} \times [0,1]) \cup ([0,1] \times \{0\})$ satisfies the geometric assumption.
\end{enumerate}

\ \par
With these particular magnetic fields, we are able to prove a local controllability result, which is similar to Theorem \ref{Theo:Bounded} (but we emphasize once again that the proofs will be rather different).
\begin{thm}
\label{Theo:Mag}
Let $b$ satisfying the geometric assumption \eqref{GeometricCondition}. 
Let $\gamma >2$ and $T>0$. There exist $\kappa, \kappa' > 0$ such that the following holds. Let
$f_{0}$ and $f_{1}$ be two functions in $C^{1}(\T^{2} \times \R^{2})
\cap W^{1,\infty}(\T^{2} \times \R^{2})$, satisfying the condition
that for any $(x,v) \in \T^{2} \times \R^{2}$ and $i \in \{ 0 , 1
\}$,
\begin{equation}
\label{DecroissanceInfiniCM}
\left\{ \begin{array}{l}
{\II |f_{i}(x,v) | \leq \kappa (1 + |v|)^{-\gamma-1},} \\
{\II | \nabla_{x} f_{i} | + | \nabla_{v} f_{i} | \leq \kappa' (1 +
|v|)^{-\gamma},}
\end{array} \right.
\end{equation}
and
\begin{equation}
\label{MemeDensiteNeutralisanteCM}
\int_{\T^{n} \times \R^{n}} f_0 = \int_{\T^{n} \times \R^{n}} f_1.
\end{equation}
Then there exists a control $G \in C^{0}([0,T] \times \T^{2} \times \R^{2})$
, such that the solution of
(\ref{VlasovCM})-(\ref{PoissonCM}) and (\ref{CondInitialeCM}) exists, is
unique, and satisfies $f(T,x,v)=f_1$.
\end{thm}
%
%
%
%
\subsection{Organization of the paper}
The paper is organized as follows:
first, in Section \ref{Sec:strategie}, we will recall some considerations on the Vlasov-Poisson equation and will explain the general strategy of the proofs. 
Then, we prove Theorem \ref{Theo:Bounded} in Section \ref{Sec:BEF} and Theorem \ref{Theo:BoundedGlobal} in Section \ref{Sec:BEFglobal}, for what concerns the bounded external field case. Finally, in Section \ref{Sec:CM}, we prove Theorem \ref{Theo:Mag} on the local controllability in the external magnetic field case.
%
\section{Strategy of the proofs}
\label{Sec:strategie}
%
%
\subsection{Notations}
%
For $T>0$, we
denote $Q_{T}:=[0,T] \times \T^{n} \times \R^{n}$, and
$\Omega_{T}:=[0,T] \times \T^{n}$. For a domain $\Omega$, we write
also $C_{b}^{l}(\Omega)$, for $l \in \N$, for the set $C^{l}(\Omega) \cap
W^{l,\infty}(\Omega)$. All the same, $C_{b}^{l+\sigma}(\Omega)$ for
$\sigma \in (0,1)$ stands for the set of $C^{l}$ functions with
bounded $\sigma$-H{\"o}lder $l$-th derivatives. Also, $C_{b}^{\sigma,
l+\sigma'}(\Omega_{T})$ (resp. $C_{b}^{\sigma, l+\sigma'}(Q_{T})$), for $l
\in \N$, $\sigma, \sigma'\in [0,1)$ is the set of continuous functions
in $\Omega_{T}$ (resp. $Q_T$), which are $C^{l}$ with respect to $x$
(resp. to $(x,v)$), and which $l$-th derivatives are all
$C^{\sigma}_{b}$ with respect to $t$ and $C^{\sigma'}_{b}$ with respect to $x$
(resp. to $(x,v)$). \par
For $x$ in $\T^{n}$ and $r >0$, we denote by $B(x,r)$ the open ball
with center $x$ and radius $r$, and by $S(x,r)$ the corresponding
sphere. The radii will always be chosen small enough in order that
$S(x,r)$ does not intersect itself (that is $r < 1/2$ in the standard
torus).

\subsection{The case $F=0$, obstructions to controllability}

In this paragraph,we focus on the case $F=0$, following \cite{OG03}. Let us consider the linearized equation around the trivial state $(\overline{f}, \overline{\Phi})=(0,0)$.
The linearized equation happens to be the free transport equation, which simply reads:
\[
\partial_t f + v.\nabla_x f = \mathbbm{1}_\omega g.
\]
By Duhamel's formula, we obtain the explicit representation for $f$:
\begin{equation}
f(t,x,v)=f_0(x-tv,v) + \int_0^t  (\mathbbm{1}_\omega g)(s,x-(t-s)v,v)ds,
\end{equation}
from which one can observe that there are two types of obstruction to controllability:
\begin{itemize}
\item[--] (small velocities) The second obstruction concerns the small velocities. The velocity of a particle can have a good direction, but if it is not high enough, then it will not be able to reach zone in the desired time, see Figure \ref{Fig2:ObstructionsVlasov}.
\item[--] (large velocities, wrong direction) The first obstruction is of geometric control type as in \cite{BLR} for what concerns the wave equation: if a particle has initially a wrong direction, then it will never reach the control zone, and thus we cannot influence its trajectory, see again Figure \ref{Fig2:ObstructionsVlasov}.
\begin{figure}[!ht]
\begin{center}
{\resizebox{10cm}{!}{\input{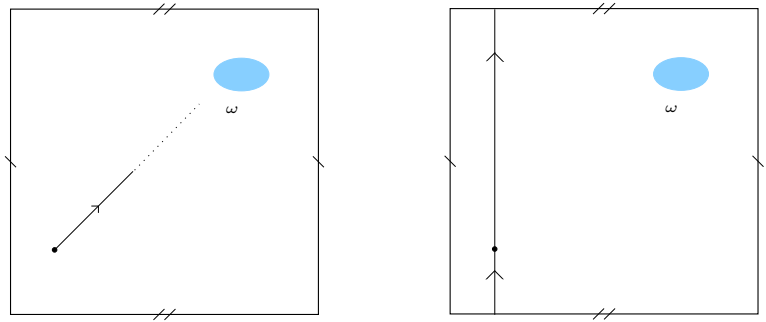}}}
\end{center}
\caption{Obstructions for small and large velocities}
\label{Fig2:ObstructionsVlasov}
\end{figure}
\end{itemize}
It follows that in general, the linearized equation fails to be controllable.
\subsection{The return method}
In order to overcome these obstructions, the idea is to use the return method, which was introduced by Coron in \cite{C92} for the study of the stabilization of finite-dimensional systems, and then used in the context of the control of PDEs by Coron in \cite{C96} for the control of the two-dimensional Euler equation for perfect incompressible fluids. It has been used since in many different contexts of PDE control: we refer to the monograph of Coron \cite{C07} for several illustrations and references for this method. The principle is to build a reference solution $(\overline{f},\overline{\Phi})$ starting from $(0,0)$ and reaching $(0,0)$ in some fixed time, and around which the linearized equation enjoys nice controllability properties. Such a construction can be delicate, and crucially depends on the structure of the studied equation.

Here, the problem is more or less equivalent to find solutions $(\overline{f},\overline{\Phi})$ (starting from $(0,0)$ and reaching $(0,0)$) and such that the characteristics associated to $\nabla \overline{\Phi}$ satisfy:
\begin{equation}
\label{eq:bonnescarac}
\forall x \in \mathbb{T}^n, \quad \forall v \in \mathbb{R}^n, \quad \exists t \in [0,T], \quad X(t,0,x,v) \in \omega.
\end{equation}
(As a matter of fact, the characteristics will not be quite associated to $\nabla \overline{\Phi}$ inside the control zone.) \par
When no exterior force is present, the existence of such a reference solution $\overline{f}$ was proved by the first author in \cite{OG03} in two dimensions, for an arbitrary control set $\omega$. This is achieved using complex analysis tools by building harmonic potentials outside $\omega$, which allow to sufficiently influence the trajectories, so that the two previous obstructions are circumvented. This strategy distinguishes between high and low velocities, for which the relevant potentials are different.
\subsection{On the scaling properties of Vlasov-Poisson equations}
We notice that (\ref{Vlasov})-(\ref{Poisson}) is
``invariant'' by some change of scales. More precisely, when $f$ is a solution
of (\ref{Vlasov})-(\ref{Poisson}) in $[0,T] \times \T^{n} \times
\R^{n}$, then for $\lambda \not = 0$, the function
\begin{equation}
\label{ChgtEchelle}
f^{\lambda}(t,x,v) := |\lambda|^{2-n} f(\lambda t,x,v/\lambda),
\end{equation}
is still a solution of (\ref{Vlasov})-(\ref{Poisson}), in
$[0,T/\lambda] \times \T^{n} \times \R^{n}$ for the following
potential
\begin{equation}
\label{ChgtEchellePhi}
\phi^{\lambda}(t,x) := \lambda^{2} \phi(\lambda t,x).
\end{equation} 
and the external force
\begin{equation}
\label{ChgtEchelleF}
F^{\lambda}(t,x,v) := \lambda^{2} F(\lambda t,x,v/\lambda).
\end{equation} 
%
%
The choice of some particular parameters $\lambda$ will be of great help for the controllability problem.

\par
\vspace{10pt}
\noindent
{\bf The choice $\lambda=-1$.}
Using (\ref{ChgtEchelle}) with $\lambda=-1$, we observe that in order to prove Theorems \ref{Theo:Bounded} and \ref{Theo:Mag},  it is sufficient 
to prove the result for the case where $f_{1}=0$ in $[\T^{n} \backslash \omega] \times \R^{n}$.
Indeed, we observe that after imposing (\ref{ChgtEchelle}) with $\lambda=-1$, the corresponding external field remains in the same class, that is, if $F$ is bounded, then $F^{\lambda=-1}$ is still bounded (resp. if  $F$ corresponds to a magnetic field satisfying the geometric condition, then $F^{\lambda=-1}$ still  corresponds to a magnetic field satisfying the fixed sign and the geometric conditions). \par
\ \par
Then one can follow the procedure that we detail below:
\begin{itemize}
\item[--] Take $f_{0}$ as initial value and $0$ (in $(\T^n \backslash
\omega) \times \R^n$) as the final one,
\item[--] Take $(x,v) \mapsto f_{1}(x,-v)$ as initial value and again
$0$ as the final one within the force field $F(T-t,x,-v)$.
\end{itemize}
each in time $T/3$. We obtain two functions $\hat{f}_{0}$ and
$\hat{f}_{1}$. Now we may consider the function $\hat{f}$ partially defined
in $Q_T$ by
\begin{equation}
\nonumber
\left\{ \begin{array}{l}
{ \hat{f}(t,x,v) = \hat{f}_{0}(t,x,v), \ \text{ in }
[0,T/3] \times \T^{n} \times \R^{n}, } \\
{ \hat{f}(t,x,v) = 0, \ \text{ in }
[T/3,2T/3] \times [\T^{n} \backslash \omega] \times \R^{n}, } \\
{ \hat{f}(t,x,v) = \hat{f}_{1}(T-t,x,-v) \ \text{ in }
[2T/3,T] \times \T^{n} \times \R^{n}. }
\end{array} \right.
\end{equation}
Then we can complete in a regular manner $\hat{f}$ inside $[T/3,2T/3] \times \omega
\times \R^{n}$, taking care to preserve for any $t$ the value of
$\int_{\T^{n} \times \R^{n}} \hat{f}(t,x,v) dxdv$.
Finally we get a relevant solution $f$.
%
%
%
%
%
For this reason, we will systematically assume that $f_1=0$ in $[\T^{n} \backslash \omega]$ for all controllability results discussed in this work.%
\par
\vspace{10pt}
\noindent
{\bf The choice $0 < \lambda \ll1$:} The choice of the parameters in such a range is useful to prove global controllability results.
As in \cite{OG03}, it will help us in particular to prove Theorem \ref{Theo:BoundedGlobal} in the bounded external field case. The principle is that when $\lambda$ is chosen small enough then $\nabla \Phi^\lambda$ has a small $L^\infty$ small (and this is also the case for $F^\lambda$), so that we can expect characteristics for $f^\lambda$ to be close to those of some well chosen relevant reference solution. This will allow us to get rid of the smallness assumption on $f_0$. Nevertheless in order to avoid concentration effects, we will need some assumptions on the characteristics associated to the reference solution.

In the magnetic field case, we observe that $F(x,v)=b(x) v^\perp$ and thus $F^\lambda(x,v)= \lambda b(x) v^\perp$. For this reason, due to our treatment of high velocities for this case, this will not allow us to prove a global result.
%
%
%
%
%
%
%
\subsection{General strategy for external force fields $F$}

Following \cite{OG03} the main steps for proving local controllability results will be:
\par
\vspace{5pt}
\noindent
{\bf Step 1.} Build a reference solution $(\overline{f},\overline{\Phi})$ of (\ref{Vlasov})-(\ref{Poisson}) with a certain control $\overline{G}$, starting from $(0,0)$ and arriving at $(0,0)$, such that the characteristics associated to $F-\nabla \overline{\Phi}$ satisfy \eqref{eq:bonnescarac}.
\par
\vspace{5pt}
\noindent
{\bf Step 2.} Build a solution $(f,\Phi)$ close to $(\overline{f},\overline{\Phi})$, taking into account the initial condition $(f_0,\Phi_0)$ and still arriving at $(0,0)$ (outside $\omega$). This is achieved using a fixed point operator involving an absorption process in the control zone. This is where we use the smallness assumption on $f_0$.
\par
\vspace{5pt}
\noindent
The treatment of Step 2. will be quite similar to that in \cite{OG03}, although a bit more technical since we will have to take into account the geometry due to $F$. The main difference is the treatment of Step 1., for which we have to propose new ideas. The strategy is the following: \par
\ \par
\noindent
{\bf Bounded force field.}  Our strategy relies on the fact that for short times, the dynamics with the external force $F$ is well approximated by the dynamics with  $F=0$. 
We recall that in \cite{OG03}, the reference solution can be constructed for any time (which can be arbitrarily small) and any control zone in the torus.
Thus, we use the construction in the case $F=0$, for very short times and a small subset of the control zone $\omega$, and using the approximation of the dynamics, this will give us a relevant reference solution.

 %
 %
 %
%
%
\ \par
\noindent
{\bf Magnetic field.} The strategy in this case can be understood in the most simple case, that is when $b$ is a positive constant. In this case, the characteristics associated to the magnetic field can be explicitly computed: these are circles, whose radius is proportional to the norm of the velocity (which is a conserved quantity). We make two crucial observations: 
 
 \begin{itemize}
 \item[--]
When the velocity is very large, the curvature of the circles are close to zero, and at least locally (that is for small times), the trajectory is well approximated by the straight lines of the free transport case. 
  \item[--]
The magnetic field has ``mixing features'', in other words it makes the velocities of particles take every value of $\S^1$, which removes the above obstruction concerning high velocities. Hence, due to this effect, at high velocity, we do not need to create any additional force field to make the particles cross the control zone. 
\end{itemize}
This means that at high velocity any subset $\omega$ of the torus automatically satisfies the geometric condition \eqref{GeometricCondition} for the caracteristics associated to the magnetic field.

In the general case, the geometric condition on $b$ allows us to make sure that the particles are sufficiently influenced by the magnetic field, so that the previous considerations will still hold.

\subsection{On the uniqueness of the solution}
In this paragraph, we briefly discuss the uniqueness question included in the above results. \par
The first point is that, if we drop the uniqueness from the conclusions of the above theorems, we can replace the assumption
\begin{equation*}
| \nabla_{x} f_{i} | + | \nabla_{v} f_{i} | \leq \kappa' (1 +
|v|)^{-\gamma},
\end{equation*}
by the weaker one
\begin{equation*}
| \nabla_{x} f_{i} | + | \nabla_{v} f_{i} | \leq \kappa'.
\end{equation*}
This is easily seen when reading the proofs below. \par
Hence the assumptions is of $\nabla f_{i}$ belonging to some weighted space is only useful for the uniqueness issue. Let us explain how one can show uniqueness under this assumption. The main point is that in this case
the solution described above satisfies
\begin{equation}
\nonumber
|\nabla_{x,v} f (t,x,v) | \leq C(f_0,f_{1}) (1+|v|)^{-\gamma},
\end{equation}
for all $t$. This follows from the construction described below, and from the estimates on $\nabla f$ in the proof. Once these estimates are obtained, the proof of uniqueness is exactly the one of Ukai-Okabe. It consists in making the difference of two potential solutions; this difference satisfies a certain transport equation with source. Then one performs an $L^{1} \cap L^{\infty}$ estimate on the solution of this equation and uses a Gronwall argument. In our case, the source term disappears when we make this difference, so one can follow \cite{UO} without change. \par
This gives the uniqueness among the solutions satisfying
\begin{gather}
\nonumber
f \in C^1([0,T] \times \T^{n} \times \R^{n}), \ 
|f| + |\nabla_{x,v} f(t,x,v) | \leq C (1+|v|)^{-\gamma} \ \text{ and } \ \nabla \phi \in L^{\infty}(0,T;W^{1,\infty}(\T^{n})).
\end{gather}
\section{Bounded external field case}
\label{Sec:BEF}
In this section, we prove Theorem \ref{Theo:Bounded}.
As already explained, the main difficulty is to build the reference solution. Then one can use the same absorption process, that was proposed in \cite{OG03}, and find a solution to the non-linear system by a similar fixed-point argument.
\subsection{Design of the reference solution for the bounded field case}
We begin with the construction of the reference solution. Accordingly to the previous strategy, we distinguish between high and low velocities.

For the large velocities, we prove the following proposition:
\begin{prop} \label{PropGV}
Let $\tau>0$ and $H \in L^{\infty}((0,\tau) \times \T^{2};\R^{2})$. Given $x_0$ in $\mathbb{T}^2$ and $r_0$ a small positive number, there exist $\varphi \in \mathcal{C}^\infty([0,\tau] \times \mathbb{T}^2; \mathbb{R})$ and $\underline{m} >0$ such that
\begin{equation}
\label{LV1}
\Delta \varphi=0 \quad \text{in} \quad [0,\tau] \times \left[ \mathbb{T}^2 \backslash \overline{B}(x_0, r_0/10)\right]
\end{equation}
\begin{equation}
\label{LV2}
\operatorname{Supp} \varphi \subset (0,\tau) \times \mathbb{T}^2
\end{equation}
and such that, if one consider the characteristics $(\overline{X},\overline{V})$ associated to the force field $H+\nabla \varphi$ then for all $m \geq \underline{m}$:
\begin{multline}
\label{LV3}
\forall x \in \mathbb{T}^2, \forall v \in \mathbb{R} ^2 \text{ such that } \vert v \vert \geq m, \ \exists t \in (\tau/3,2\tau/3), \\
\text{such that } \ \overline{X}(t,0,x,v) \in B(x_0,r_0/4) \text{ and } |\overline{V}(t,0,x,v)| \geq \frac{m}{2}.
\end{multline}
\end{prop}
\begin{proof}[Proof of Proposition \ref{PropGV}]
In the case $H=0$, this proposition was already proved in \cite[Proposition 1, p. 340]{OG03}.
 We fix $x'_0=x_0$, $r_0'=r_0/2$. Applying this result for $\tau=1$, we thus obtain the existence of $\varphi_1  \in \mathcal{C}^\infty([0,1] \times \mathbb{T}^d; \mathbb{R})$ and $m'\in \mathbb{R}^{+*}$ with compact support in time in $(0,1)$, satisfying:
\begin{equation}
\Delta \varphi=0 \quad \text{in} \, [0,1] \times \left[ \mathbb{T}^2 \backslash \overline{B}(x_0, r_0/20)\right],
\end{equation}
\begin{equation}
\operatorname{Supp} \varphi \subset (0,1) \times \mathbb{T}^2,
\end{equation}
and such that, if one consider the characteristics $(\tilde X^1, \tilde V^1)$ associated to the force field $\nabla \varphi_1$ then:
\begin{equation}
\forall x \in \mathbb{T}^d, \forall v \in \mathbb{R} ^d, \, \text{such that} \, \vert v \vert \geq m, \exists t \in (1/4, 3/4), \tilde X^1(t,0,x,v) \in B(x_0,r_0/8).
\end{equation}
Let $\tau'<\tau$ to be fixed later. For this given $\tau'$, we can construct $\varphi_{\tau'}$ by rescaling $\varphi_{1}$ as follows:%
\begin{equation} \label{PhiTauPrime}
\varphi_{\tau'}(t,x) := \frac{1}{(\tau')^{2}} \varphi_{1} \left(\frac{t}{\tau'},x\right),
\end{equation}
which corresponds to follow the characteristics with time $\frac{t}{\tau'}$. 
\par
\par
Now let us consider the shifted in time potential $\varphi$ defined by:
\begin{equation}
\varphi(t,x)= \varphi_{\tau'}\left(t-\frac{\tau-\tau'}{2},x\right).
\end{equation}
We extend $\varphi$ by $0$ in $(0,\tau)\setminus \left(\frac{\tau-\tau'}{2},\frac{\tau+\tau'}{2}\right)$.

We define the characteristics $(\tilde X, \tilde V)$ associated to the force field $\nabla \varphi$, which satisfy by construction:
\begin{equation}
\forall x \in \mathbb{T}^d, \forall v \in \mathbb{R} ^d, \, \text{such that} \, \vert v \vert \geq m, \exists t \in (\frac{\tau-\tau'}{2}, \frac{\tau+\tau'}{2}), \tilde X(t,0,x,v) \in B(x_0,r_0/8).
\end{equation}

Let us now compare $(\tilde X, \tilde V)$  and $(\overline X, \overline V)$, which is associated to the force field $H+\nabla \varphi$  on $(0,\tau)$. By Taylor's formula we have:
\begin{equation}
\begin{split}
\vert \overline{X}(t,\frac{\tau-\tau'}{2},x,v) -\tilde X(t,\frac{\tau-\tau'}{2},x,v) \vert \leq \int_\frac{\tau-\tau'}{2}^t (t-s)\Big[ &\vert \nabla \varphi(s, \tilde X(s,\frac{\tau-\tau'}{2},x,v))- \nabla \varphi(s,\overline X(s,\frac{\tau-\tau'}{2},x,v))\vert \\ +& \vert H(s, \tilde X(s,\frac{\tau-\tau'}{2},x,v), \tilde V(s,\frac{\tau-\tau'}{2},x,v)) \vert \Big]ds.
\end{split}
\end{equation}
By Gronwall lemma we deduce for $t \in \left(\frac{\tau-\tau'}{2},\frac{\tau+\tau'}{2}\right)$:
\begin{equation}
\begin{split}
\vert \overline{V}(t,\frac{\tau-\tau'}{2},x,v) -\tilde V(t,\frac{\tau-\tau'}{2},x,v) \vert \leq \tau' \| H \|_{L^\infty_{t,x,v}} e^{\frac{{\tau'}^2}{2} \| \nabla^2 \varphi \|_{L^\infty((0,\tau) \times \mathbb{T}^d)}}, \\
\vert \overline{X}(t,\frac{\tau-\tau'}{2},x,v) -\tilde X(t,\frac{\tau-\tau'}{2},x,v) \vert \leq \frac{{\tau'}^2}{2} \| H \|_{L^\infty_{t,x,v}} e^{\frac{{\tau'}^2}{2} \| \nabla^2 \varphi \|_{L^\infty((0,\tau) \times \mathbb{T}^d)}}.
\end{split}
\end{equation}
The crucial point is now to observe that $\varphi$ described above satisfies:
\[
\| \nabla^2 \varphi \|_{L^\infty((0,\tau) \times \mathbb{T}^d)} =\mathcal{O}\left(\frac{1}{{\tau'}^2}\right) \ \text{ as } \tau' \rightarrow 0,
\]
as it can be seen from \eqref{PhiTauPrime}.

Thus for $\tau'$ small enough we infer that $\overline{X}(t,0,x,v)$ meets $B(x_0, r_0/4)$ for some $t \in \left(\frac{\tau-\tau'}{2},\frac{\tau+\tau'}{2}\right) \subset(\tau/3,2\tau/3)$, for all $x$ and $v$, provided that $|\overline{V}(\frac{\tau-\tau'}{2},0,x,v) |$ is large enough. This is ensured if $|v | \geq \underline{m}$ is chosen large enough, thanks to the inequality:
\[
|\overline{V}(\frac{\tau-\tau'}{2},0,x,v) |\geq |v | - \frac{\tau-\tau'}{2} \Vert H \Vert_{L^\infty_{t,x,v}}.
\]
\end{proof}
\begin{rque}
In this proof, this is crucial that $H \in  L^{\infty}((0,\tau) \times \T^{2};\R^{2})$. Thus this approach will fail for the magnetic field case.
\end{rque}
The above proposition shows that with a suitable electric potential, all particles having a sufficiently high velocity will eventually reach $\omega$. The following proposition explains how one can accelerate all particles in order to make all the remaining ones also reach $\omega$.
This will also rely on the construction in the case $F=0$. 
\begin{prop}
\label{PropAccelerePartout}
Let $\tau>0$, $M>0$ and $H \in L^{\infty}((0,\tau) \times \T^{2};\R^{2})$. Given $x_0$ in $\mathbb{T}^2$ and $r_0$ a small positive number, there exists $\tilde{M}>0$, ${\mathcal E} \in C^{\infty}([0,\tau] \times \T^2;\R^{2})$ and  $\varphi \in C^{\infty}([0,\tau] \times \T^2;\R)$ satisfying
\begin{gather}
\label{EetPhi}
{\mathcal E} = \nabla \varphi \text{ in  } [0,\tau] \times (\T^2 \backslash B(x_{0},r_{0})), \\
\label{Phi2SupportTemps}
\mbox{Supp}({\mathcal E}) \subset (0,\tau) \times \T^2, \\
\label{Phi2Harmonique}
\Delta \varphi =0 \text{ in  } [0,\tau] \times (\T^2 \backslash B(x_{0},r_{0})), 
\end{gather}
such that if $({X},{V})$ are the characteristics corresponding the force
\begin{equation} \label{Faccelere}
\mathcal{I}:= {\mathcal E} + H,
\end{equation}
then
\begin{equation} 
\label{AccelerePartout}
\forall (x,v) \in \T^{2} \times B(0,M), \ {V}(\tau,0,x,v)) \in B(0,\tilde{M}) \setminus B(0,M+1).
\end{equation}
\end{prop}
\begin{proof}[Proof of Proposition \ref{PropAccelerePartout}]
By \cite[Lemma 3, p. 356]{OG03}, there exists $\theta \in C^{\infty}(\T^{2};\R)$ such that
\begin{gather*}
\Delta \theta =0 \ \text{ in } \ \T^{2} \setminus B(x_{0},r_{0}), \\
| \nabla \theta (x) | >0 \ \text{ in } \ \T^{2} \setminus B(x_{0},r_{0}).
\end{gather*}
From the second condition, one sees that $\mbox{Ind}_{S(x_{0},r_{0})} (\nabla \theta)=0$, so that $\nabla \theta_{|\T^{2} \setminus B(x_{0},r_{0})}$ can be extended to $\T^{2}$ as a smooth non-vanishing vector field, let us say $W$. Call $\Lambda \in C^{\infty}_{0}((0,1);\R)$ a nonnegative function with $\int_{0}^{1} \Lambda =1$. We claim that for sufficiently small $\tau'<\tau$, and sufficiently large $\mathcal{C} >0$,
\begin{equation*}
{\mathcal E}(t,x):= \frac{{\mathcal C}}{\tau'} \Lambda \left( \frac{t}{\tau'} \right) W(x),
\end{equation*}
is convenient. Then all properties above but \eqref{AccelerePartout} are clear. \par
Call $(\overline{X}, \overline{V})$ the characteristics associated to ${\mathcal E}$ only. We see that for all $(x,v) \in \T^{2} \times B(0,M)$ and $t \in [0,\tau']$, 
\begin{equation*}
| \overline{V}(t,0,x,v) -v | \leq {\mathcal C} \| {\mathcal E} \|_{\infty}, \
| \overline{X}(t,0,x,v) -x | \leq \tau' ( {\mathcal C} \|  {\mathcal E} \|_{\infty} +M ), \
\end{equation*}
so
\begin{equation*}
| \overline{V}(\tau',0,x,v) -v + {\mathcal C} {\mathcal E}(x) | \leq \tau' \| {\mathcal E} \|_{\sigma} [\tau' ({\mathcal C}  \| {\mathcal E} \|_{\infty} +M )].
\end{equation*}
Noting that, due to the time support of $\mathcal{E}$,  $\overline{V}(\tau,0,x,v)=\overline{V}(\tau',0,x,v)$ and using that $|{\mathcal E}| \geq c >0$ on $\T^{2}$, one sees that one can choose ${\mathcal C}$ and then $\tau'$ such that
\begin{equation*}
\forall (x,v) \in \T^{2} \times B(0,M), \  \overline{V}(\tau,0,x,v) \in \R^{2} \setminus B(0,M+2 + \tau \| H \|_{\infty}).
\end{equation*}
We now consider the characteristics $(X, V)$ associated to ${\mathcal E} + H$ and evaluate:
\begin{equation*}
\vert \overline{X}(t,0,x,v) - X(t,0,x,v) \vert \leq \int_0^t \vert \overline{V}(s,0,x,v) - V(s,0,x,v) \vert ds
\end{equation*}
\begin{equation}
\begin{split}
| \overline{V}(t,0,x,v) - V(t,0,x,v) | \leq 
& \int_0^t \Big( | {\mathcal E} (s, \overline{X}(s,0,x,v)) -{\mathcal E} (s,  X(s,0,x,v)) | \\ 
+& | H(t,  X(s,0,x,v),  V(s,0,x,v) ) | \Big) ds \\
\leq& \| \nabla {\mathcal E} \|_{L^\infty((0,\tau') \times \mathbb{T}^2)} \int_0^t (t-s) | \overline{V}(s,0,x,v) - V(s,0,x,v) | ds 
+ t \| H \|_{L^\infty_{t,x,v}}. 
\end{split}
\end{equation}
By Gronwall's inequality:
\begin{equation}
\label{3.16}
\vert \overline{V}(t,0,x,v) - V(t,0,x,v) \vert \leq t \| H \|_{L^\infty_{t,x,v}}  e^{\frac{t^2}{2} \| \nabla {\mathcal E} \|}
\end{equation}
We observe that we have: 
\begin{equation}
\label{scalingphitemps}
\frac{\tau'^2}{2} \| \nabla {\mathcal E} \|_{L^\infty((0,\tau') \times \mathbb{T}^2)}  =\mathcal{O}(\tau') \ \text{ as } \tau' \rightarrow 0.
\end{equation}
Taking $\tau'$ small enough, using $t= \tau'$ in \eqref{3.16},  and observing that
\begin{equation*}
|V(\tau,0,x,v) - V(\tau',0,x,v)| \leq |\tau - \tau'| \| H \|_{\infty},
\end{equation*}
allow us to prove our claim. The existence of $\tilde{M}$ is a matter of compactness of $\T^{2} \times \overline{B}(0,M+2 + \tau \| H \|_{\infty})$.
\end{proof}
\begin{rque}
We can observe that there is some ``margin'' in the previous proof, in the sense that if we only had
\[
\frac{\tau'^2}{2} \| \nabla {\mathcal E} \|_{L^\infty((0,\tau') \times \mathbb{T}^2)}  =\mathcal{O}(1) \ \text{ as } \tau' \rightarrow 0,
\]
the proof would still follow. However, that \eqref{scalingphitemps} holds will actually be crucial in the proof of the equivalent lemma in the magnetic field case, and this time this will be sharp.
\end{rque}
\ \par
\noindent
{\bf The reference solution.} Now we are able to define the reference solution. Consider $x_0$ in $\omega$ and $r_0$ a small positive number such that
\[
B(x_0,2r_0) \subset \omega.
\]
We first define a reference potential $\overline{\phi}:[0,T] \times \T^{2} \rightarrow \R$ as follows. We apply Proposition \ref{PropGV} with $\tau =T/3$, $H=F_{|[0,T/3]}$, we obtain $\overline{\phi}_{1}$ and some $\underline{m}_1>0$ such that \eqref{LV3} is satisfied.
\par
Let
\begin{equation}
\alpha=\max\left(\frac{600 r_{0}}{T}, C_{r_0} (1+\| F \|_{\infty} + \| \overline{\varphi}_{1} \|_{\infty}  + \| \overline{\varphi}_{3} \|_{\infty} )\right),
\end{equation}
\begin{equation}\label{DefM}
M_1=\max (\underline{m}_1, 2\alpha )+\frac{T}{3}\left(\| \nabla \overline{\phi}_{1} \|_{\infty} + \| F \|_{\infty} \right), \quad M_2=\max (\underline{m}_3, 2\alpha), \quad M= \max(M_1,M_2).
\end{equation}
Above $C_{r_0}$ is a positive geometric constant depending only on $r_0$, and which will be described later.

We also use Proposition \ref{PropGV} again with $\tau =T/3$, $H(t,x)=F(t+ \frac{2T}{3},x)$ for $t \in [0,T/3]$, we obtain $\overline{\phi}_{3}$ and some $\underline{m}_{3}>0$ such that \eqref{LV3} is satisfied.  Then we apply Proposition \ref{PropAccelerePartout} with $\tau=T/3$,  $H(t,x)=F(t+ \frac{T}{3},x)$  for $t \in [0,T/3]$, and $M$ described above.
We obtain $\overline{\mathcal E}_{2}$ $\overline{\phi}_{2}$ and some $\tilde{M}$. 

Finally we set:
\begin{equation*}
\overline{\phi}(t,\cdot) = \left\{ \begin{array}{l}
 \overline{\phi}_{1}(t,\cdot) \text{ for } t \in [0,\frac{T}{3}], \\
 \overline{\phi}_{2}(t - \frac{T}{3},\cdot) \text{ for } t \in [\frac{T}{3},\frac{2T}{3}], \\
 \overline{\phi}_{3}(t - \frac{2T}{3},\cdot) \text{ for } t \in [\frac{2T}{3},T],
\end{array} \right.
\end{equation*}
and
\begin{equation*}
\overline{{\mathcal E}}(t,\cdot) = \left\{ \begin{array}{l}
\nabla \overline{\phi}_{1}(t,\cdot) \text{ for } t \in [0,\frac{T}{3}], \\
 \overline{{\mathcal E}}_{2}(t - \frac{T}{3},\cdot) \text{ for } t \in [\frac{T}{3},\frac{2T}{3}], \\
\nabla \overline{\phi}_{3}(t - \frac{2T}{3},\cdot) \text{ for } t \in [\frac{2T}{3},T].
\end{array} \right.
\end{equation*}
Let us now introduce $\overline{f}$.
Consider a function $\mathcal Z \in C^{\infty}_0(\R^{n} ; \R)$
satisfying the following constraints
\begin{equation}
\label{DefZ}
\left\{ \begin{array}{l}
{{\mathcal Z} \geq 0 \text{ in } \R^{n},} \\
{\mbox{Supp } {\mathcal Z} \subset B_{\R^{n}}(0,1), }\\
{\displaystyle{ \int_{\R^{n}} {\mathcal Z} =1.   }}
\end{array} \right.
\end{equation}
We introduce $\overline{f}=\overline{f}(t,x,v)$ as
\begin{equation}
\label{Deffbar}
\overline{f}(t,x,v):= {\mathcal Z}(v) \Delta \overline{\varphi}(t,x).
\end{equation}
Of course, $\overline{f}$ satisfies (\ref{Vlasov}) in $[0,T]
\times\T^{2} \times \R^{2}$, with source term
\begin{equation}
\label{DefGbar}
\overline{G}(t,x,v):= \partial_{t} \overline{f} + v.\nabla_{x}
\overline{f} + (F + \nabla \overline{\varphi}).\nabla_{v} \overline{f},
\end{equation}
which is supported in $[0,T] \times B(x_{0},r_{0}) \times
\R^{2}$. Up to an additive function of $t$, the function $\varphi$
satisfies the equation (\ref{Poisson}) corresponding to $\overline{f}$
(with $\overline{f}(0,\cdot,\cdot) \equiv 0$).
We denote 
\begin{equation*}
\overline{\rho}(t,x):=\int_{\R^{2}} \overline{f}(t,x,v) \, dv =\Delta \overline{\varphi}(t,x). 
\end{equation*}
\subsection{Fixed point operator}
\label{sectionFP}
To prove Theorem \ref{Theo:Bounded}, we construct directly the solution $f$ starting at $f_{0}$ and reaching $0$ in $\T^{2} \setminus \omega$ at time $T$, provided that $f_{0}$ is suitably small. This is done by a fixed-point procedure. In this subsection, we describe the operator; in the next ones, we will find a solution to our controllability problem as a fixed point of this operator. \par
\ \par
Let $\varepsilon \in (0,1)$.
We first define the domain ${\mathcal S}_{\varepsilon}$ of
$V_{\varepsilon}$ by
\begin{gather}
\nonumber
\begin{array}{ll}
{{\mathcal S}_{\varepsilon}:= 
\left\{ \II \ g \in C_{b}^{\delta_2}(Q_{T}) \ \right/ \hfill} & {\hfill }
\end{array} \\
\label{DefS}
\begin{array}{ll}
{\mathbf a.\ } & {\| \int_{\R^{2}} (g - \overline{f}) \, dv
\|_{C^{\delta_1}(\Omega_{T})} \leq \varepsilon,} \\
{\mathbf b.\ } &
{\left. \| (1+ |v|)^{\gamma}( g - \overline{f})\|_{L^{\infty}(Q_{T})} 
\leq c_{1} 
\left[ \II \|f_{0}\|_{C_{b}^{1}(\T^{2} \times \R^{2})}
+\|(1+|v|)^{\gamma}f_{0}\|_{C_{b}^{0}(\T^{2} \times \R^{2})} \right], 
\right. } \\
{\mathbf c.\ } &
{\left. \| g - \overline{f} \|_{C_{b}^{\delta_2}(Q_{T})} \leq c_2
\left[ \II \|f_{0}\|_{C_{b}^{1}(\T^{2} \times \R^{2})}
+\|(1+|v|)^{\gamma}f_{0}\|_{C_{b}^{0}(\T^{2} \times \R^{2})} \right] ,
\right.} \\
{\mathbf d.\ } &
{\left. \II \forall t \in [0,T],\ \int_{\T^{2} \times \R^{2}}  g(t,x,v)\,  dx \, dv = \int_{\T^{2} \times \R^{2}} f_{0}(x,v) 
dx dv\ \right\} ,}
\end{array}
\end{gather}
with $c_1$, $c_2$ depending only on $\gamma$, $T$, $\omega$ (and
hence on $(\overline{f},\overline{\varphi})$) and $F$, but not on $\varepsilon$. The
indices $\delta_1 < \delta_2$ in $(0,1)$ are fixed as follows
\begin{equation}
\label{CalculDesDeltas}
\delta_1:= \frac{\gamma-n}{2(\gamma+1)} \text{ and }  \delta_2:= \frac{\gamma}{\gamma+1}.
\end{equation}
For fixed $c_1$ and $c_2$ large enough depending only on
$(\overline{f},\varphi)$, and $f_0$ small enough, one has 
\begin{equation*}
\left|\int f_0 dvdx \right| \leq \varepsilon,
\end{equation*}
and consequently, in this case $f_0+\overline{f} \in {\mathcal S}_\varepsilon$, so ${\mathcal
S}_\varepsilon \not = \emptyset$. From now, this is systematically supposed to be the
case. \par
Now we introduce the following subsets of $S(x_0,r_0)
\times \R^{2}$:
\begin{equation}
\label{DefGammaMoins} 
\gamma^{-}:= \left\{ \II (x,v) \in S(x_0,r_0) \times
\R^{2} \ /\ |v| > \frac{1}{2} \text{ and } v.\nu(x) <-\frac{1}{10}
|v|  \right\},
\end{equation}
\begin{equation}
\label{DefGamma32Moins} 
\gamma^{2-}:= \left\{ \II (x,v) \in S(x_0,r_0) \times
\R^{2} \ /\ |v| \geq 1 \text{ and } v.\nu(x) \leq -\frac{1}{8}
|v| \right\}.
\end{equation}
\begin{equation}
\label{DefGammaMoinsMoins} 
\gamma^{3-}:= \left\{ \II (x,v) \in S(x_0,r_0) \times
\R^{2} \ /\ |v| \geq 2 \text{ and } v.\nu(x) \leq -\frac{1}{5}
|v| \right\},
\end{equation}
\begin{equation}
\label{DefGammaPlus} 
\gamma^{+}:= \left\{ \II (x,v) \in S(x_0,r_0) \times
\R^{2} \ /\ v.\nu(x) \geq 0 \right\},
\end{equation}
where $\nu(x)$ stands for the unit outward normal to the sphere $S(x_0,r_0)$ at point $x$.
It can be easily seen that
\begin{equation}
\nonumber
\mbox{dist}( [S(x_0,r_0) \times \R^2] \backslash \gamma^{2-}; \gamma^{3-}) >0.
\end{equation}
\ \\
We introduce a $C^{\infty} \cap C_{b}^{1}$ regular function $U:S(x_0,r_0)
\times \R^2 \to \R$, satisfying
\begin{equation}
\label{DefU}
\left\{ \begin{array}{l}
{ 0 \leq U \leq 1, } \\
{U \equiv 1 \text{ in } [S(x_0,r_0) \times \R^2] \backslash
\gamma^{2-},} \\
{U \equiv 0 \text{ in } \gamma^{3-}.}
\end{array} \right.
\end{equation}
We also introduce a function $\Upsilon:\R^{+} \to \R^{+}$, of
class $C^{\infty}$, such that
\begin{equation}
\label{DefUpsilon}
\Upsilon = 0 \text{ in } \left[0,\frac{T}{48}\right] \cup  \left[\frac{47T}{48},T\right]  \ \text{ and } \
\Upsilon = 1 \text{ in } \left[\frac{T}{24},\frac{23T}{24}\right].
\end{equation} \par
Now, given $g \in {\mathcal S}_{\varepsilon}$, we associate $\phi^g$ on $[0,T] \times \T^2$
by
\begin{equation}
\label{defphiPoisson}
\left\{ \begin{array}{l} { \Delta \phi^{g} (t,x) =  \int_{\R^{n}} g(t,x,v) \, dv - \int_{\T^n \times \R^{n}} g(t,x,v) \, dv \, dx \text{ in } [0,T] \times \T^{n}, } \\
\int_{\T^{n}} \phi^{g}(t,x) \, dx = 0 \text{ in } [0,T]. 
\end{array} \right.
\end{equation} \par
Then, we define $\tilde{V}(g):=f$ to be the solution of the following system
\begin{equation}
\label{EqLin3}
\left\{ \begin{array}{l}
f(0,x,v) = f_{0} \text{ on } \T^{2} \times \R^{2}, \medskip \\
\partial_{t} f + v.\nabla_{x} f + (F+\nabla \phi^{g} + \overline{\mathcal E} - \nabla \overline{\varphi}).\nabla_{v} f = 0
\text{ in } [0,T] \times [(\T^{n} \times \R^{n}) \backslash \gamma^{-}], \medskip \\
f(t,x,v) =[ 1- \Upsilon(t)] f(t^{-},x,v) + \Upsilon (t) U(x,v) f(t^{-},x,v) \text{ on } [0,T] \times \gamma^{-}.
\end{array} \right. 
\end{equation} 
To explain the last equation, we introduce the characteristics $(X,V)$ associated to the force field $F+\nabla \phi^{g} + \overline{\mathcal E} - \nabla \overline{\varphi}$. 
In the previous writing, $f(t^{-},x,v)$ is the limit value of $f$ on
the characteristic $(X,V)(s,t,x,v)$ as the time $s$ goes to
$t^-$. (For times before $t$, but close to $t$, the corresponding
characteristic is not in $\gamma^{-}$.) 
When the characteristics $(X,V)$ meet $\gamma^-$ at time $t$, then
the value of $f$ at time $t ^+$ is fixed according to the last
equation in (\ref{EqLin3}). One can see the function $\Upsilon (t) U(x,v)$ as an opacity 
factor which varies according to time and to the incidence of the characteristic on $S(x_{0},r_{0})$. 
In this process a part of $f$ is absorbed on $\gamma^{-}$,
which varies from the totality of $f$ to no absorption according to the angle of incidence, the modulus of the velocity and the time. \par
The set of times when a characteristic meets $\gamma^- $ is discrete. Indeed, if $(X,V)(t,0,x,v) \in \gamma^{-}$ and $(X,V)(t',0,x,v) \in \gamma^{-}$, then there exists $s \in (t,t')$ for which $(X,V)(s,0,x,v) \in \gamma^{+}$. The conclusion follows from $\mbox{dist}(\gamma^{+},\gamma^{-})>0$. \par
\ \par
We now consider a continuous linear extension operator
$\overline{\pi}:C^{0}(\T^{2}\backslash B(x_{0},2r_{0});\R) \to
C^{0}(\T^{2};\R)$, and which has the property that each
$C^{\alpha}$-regular function is continuously mapped to a
$C^{\alpha}$-regular function, for any $\alpha \in [0,1]$. \par
From this operator, we deduce a new one $\tilde{\pi}:C^{0}((\T^{n} \backslash
B(x_{0},2r_{0})) \times \R^{n}) \to C^{0}(\T^{n} \times \R^{n})$
according to the rule:
\begin{equation}
\label{ReglePi}
(\pi f)(x,v):= [\overline{\pi}f(\cdot,v)](x).
\end{equation}
Then we modify this operator in order to get the further property that for any integrable $f \in
C^{0}((\T^{n}\backslash B(x_{0},2r_{0})) \times \R^{n})$, one has
\begin{equation}
\label{NeutralitedePi}
\int_{\T^{n} \times \R^{n}} {\pi}(f) \,  dv \, dx = \int_{\T^{n} \times \R^{n}}
f_{0}(x,v)\,  dv \, dx.
\end{equation}
This condition can easily be obtained by considering a regular, compactly supported,
nonnegative function $u$ with integral $1$ in $B(x_{0},r_{0}) \times \R^{n}$, and adding to ${\pi}(f)$ the function 
\begin{equation*}
\left[\int_{\T^{n} \times \R^{n}} f_{0} -\int_{(\T^{n} \setminus \omega) \times \R^{n}} f\right] u.
\end{equation*}
We obtain a continuous affine operator $\pi$ satisfying that for some constant $c_{\pi}$, one has for any integrable $f \in C^{1}(\T^{2}\backslash B(x_{0},2r_{0}))$, one has
\begin{gather*}
\label{EstPi}
\| {\pi}(f) \|_{C_{b}^{1}} \leq c_{\pi} \| f \|_{C_{b}^{1}}
+ \left|\int_{(\T^n \setminus \omega) \times \R^n} f - \int_{\T^n \times \R^n} f_0 \, dv \, dx \right|,  \\
\| {\pi}(f) \|_{L^{\infty}} \leq c_{\pi} \| f \|_{L^{\infty}}.
+ \left|\int_{(\T^n \setminus \omega) \times \R^n} f - \int_{\T^n \times \R^n} f_0 \, dv \, dx \right|.
\end{gather*}
Due to the compact support of $u$, $\pi$ continuously sends $L^{\infty}((\T^{n} \setminus \omega) \times \R^{n}; (1+|v|)^{\gamma}\,dx)$ into  $L^{\infty}(\T^{n} \times \R^{n}; (1+|v|)^{\gamma}\,dx)$, with estimates as above. \par
It is convenient to introduce another truncation in time function $\tilde{\Upsilon}$ such that:
\begin{equation}
\label{DefUpsilon2}
\tilde{\Upsilon} = 0 \text{ in } \left[0,\frac{T}{100}\right]  \ \text{ and } \
\tilde{\Upsilon} = 1 \text{ in } \left[\frac{T}{48},T\right].
\end{equation} \par
Finally, we introduce the operator $\Pi :C^{0}(([0,T] \times \left[
\T^{2} \backslash B(x_{0},2r_{0}) \right] \times \R^{2}) \cup
([0,T/48] \times \T^{2} \times \R^2)) \to C^{0}([0,T] \times \T^{2}
\times \R^{2})$ given by:
\begin{equation}
\label{RegleGrandPi}
(\Pi f)(t,x,v):= (1-\tilde{\Upsilon}(t)) f(t,x,v)
+ \tilde{\Upsilon}(t) [{\pi}f(t,\cdot,\cdot)](x,v).
\end{equation}
We finally define ${\mathcal V}[g]$ by:
\begin{equation}
\label{DefV2}
{\mathcal V}[g]:= \overline{f} + \Pi ( f_{|([0,T] \times \left[
\T^{2} \backslash B(x_{0},2r_{0}) \right] \times \R^{2}) \cup
([0,T/48] \times \T^{2} \times \R^2)} ) \text{ in } [0,T] \times \T^{2} \times \R^{2}.
\end{equation} \par
\subsection{Existence of a fixed point}
The goal of this paragraph is to prove the existence of a fixed point for small values of $\epsilon$, which corresponds to the following lemma.
\begin{lem} \label{Lem:ExistFP}
There exists $\epsilon_0>0$ such that for any $0<\epsilon<\epsilon_0$, there exists a fixed point of ${\mathcal V}$ in $\mathcal{S}_\epsilon$.
\end{lem}
The proof is almost the same as in \cite[Section 3.3]{OG03}. In order to avoid to repeat it, we only give the main arguments and refer to it for the details. We only focus on the main differences. \par
We endow the domain ${\mathcal S}_{\varepsilon}$ with the norm of $C^{0}([0,T] \times \T^{2} \times \R^{2})$. The existence of a fixed point of ${\mathcal V}$ on ${\mathcal S}_{\varepsilon}$ relies on Schauder's theorem. Accordingly, we have to prove that ${\mathcal S}_{\varepsilon}$ is a convex compact subset of $C^{0}([0,T] \times \T^{2} \times \R^{2})$, that ${\mathcal V}$ is continuous on ${\mathcal S}_{\varepsilon}$ for this topology, and finally that ${\mathcal V}({\mathcal S}_{\varepsilon}) \subset {\mathcal S}_{\varepsilon}$. \par
That ${\mathcal S}_{\varepsilon}$ is convex is clear; that it is compact follows from Ascoli's theorem, using both uniform H\"older estimates and the uniform weighted estimates. \par
\ \par
Now let us discuss the continuity of ${\mathcal V}$. Here the proof of \cite[Section 3.3]{OG03} actually holds without further modification. Let us briefly explain the argument. Due to the compactness of ${\mathcal S}_{\varepsilon}$, it is sufficient to prove that if $f_{n} \rightarrow f$ in ${\mathcal S}_{\varepsilon}$, then ${\mathcal V}[f_{n}] \rightarrow {\mathcal V}[f]$ pointwise. Let us fix $(x,v) \in \T^{2} \times \R^{2}$. Call $(X^{n},V^{n})$ and $(X,V)$ the characteristics associated to the force $F + \nabla \phi^{f_{n}}$ and $F + \nabla \phi^{f}$, respectively. By Gronwall's lemma, $(X^{n},V^{n})$ converges to $(X,V)$ uniformly on compacts. \par
If there was no absorption (that is, if we took $U=0$), then the convergence \[{\mathcal V}[f_{n}](t,x,v) \rightarrow {\mathcal V}[f](t,x,v)\] would follow from $\nabla \phi^{f_{n}} \rightarrow \nabla \phi^{f}$ uniformly on $[0,T] \times \T^{2}$ and Gronwall's lemma. The difficulty comes from the fact that we have to take into account in ${\mathcal V}[f](t,x,v)$ the various times of absorption on $\gamma^{-}$. But from the convergence of $(X^{n},V^{n})$ to $(X,V)$ (uniformly on compacts), one can deduce that for $n$ large enough, $(X^{n},V^{n})(\cdot,0,x,v)$ meets $\gamma^{-}$ the same number of times as $(X,V)(\cdot,0,x,v)$, and that the intersection points of $(X^{n},V^{n})(\cdot,0,x,v)$ and $\gamma^{-}$ converge towards those of $(X,V)(\cdot,0,x,v)$. Then the continuity of ${\mathcal V}$ follows. \par
\ \par
The main point in the proof is to establish that ${\mathcal V}({\mathcal S}_{\varepsilon}) \subset {\mathcal S}_{\varepsilon}$. The crucial estimate here is the following.
\begin{lem} \label{LemCrucial}
Let $g \in {\mathcal S}_{\varepsilon}$, and $(X,V)$ the characteristics associated to $F + \nabla \phi^{g}$. Then one has
\begin{equation} \label{CompareModuleVitesse}
\big| |v|- |V(t,0,x,v)| \big| \leq 1 + t \| F + \nabla \phi^{g} \|_{\infty}.
\end{equation}
\end{lem}
This lemma is trivial in the case under view, even with $|v - V(t,0,x,v)|$ on the left hand side. But since the estimate with $|v - V(t,0,x,v)|$ on the left hand side is not valid in the presence of a magnetic field, we prefer to use \eqref{CompareModuleVitesse}. \par
\ \par
Let $g \in {\mathcal S}_{\varepsilon}$. That the point {\bf d.} is true for ${\mathcal V}[g]$ comes from the construction, in particular from the choice of the operator $\Pi$ (see \eqref{NeutralitedePi}). \par
\ \par
Let us explain why the point {\bf b.} is satisfied by $f:=\tilde{{\mathcal V}}[g]$. From the construction, on $\gamma^{-}$ one has $|f(t^{+},x,v)| \leq |f(t^{-},x,v)|$. It follows that
\begin{equation}
\nonumber
|f(t,x,v)| \leq |f_{0}[(X,V)(0,t,x,v)]|.
\end{equation}
Now,
\begin{eqnarray*}
|f(t,x,v)|
&\leq& \|(1+|v|)^{\gamma}f_{0}\|_{L^{\infty}} \Big[ 1 + \big| |v| - \big(|v| - |V(0,t,x,v)| \big) \big| \Big]^{-\gamma} \\
&\leq& \|(1+|v|)^{\gamma}f_{0}\|_{L^{\infty}} \left( \frac{1+ \big| |v| - |V(0,t,x,v)| \big|}{1+|v|} \right)^{\gamma},
\end{eqnarray*}
where we used
\begin{equation*}
(1+|x-x'|)^{-1} \leq \frac{1+|x|}{1+|x'|}.
\end{equation*}
Note that $\| F + \nabla \varphi^{g} \|_{\infty} \leq \| F\|_{\infty}  + \varepsilon \leq \| F\|_{\infty}  + 1$.
With Lemma \ref{LemCrucial}, we deduce that for some $C>0$ independent of $f_{0}$ and $\varepsilon$:
\begin{equation}
\nonumber
|(1+|v|)^{\gamma} f(t,x,v)| \leq  C \|(1+|v|)^{\gamma}f_{0}\|_{L^{\infty}} .
\end{equation}
Then the fact that ${\mathcal V}[g]$ also satisfies {\bf b.} follows from the construction of the operator $\Pi$.
 \ \par
\ \par
Let us now explain the point {\bf c.} We have the following lemma:
\begin{lem}
\label{LemmeRegHV}
For $g \in {\mathcal S}_\varepsilon$, one has $\tilde{{\mathcal V}}[g] \in
C^{1}(Q_{T} \backslash \Sigma_{T})$, with $\Sigma_{T}:= [0,T] \times \gamma^{-}$.
Moreover, for any $(t,x,v)$ and $(t',x',v')$ in $[0,T]
\times [ \T^{2} \backslash \omega ] \times \R^{2}$, with $|
v-v' | \leq 1$, one has,
\begin{multline}
\label{EstGradV2}
| \tilde{{\mathcal V}}[g] (t,x,v) - \tilde{{\mathcal V}}[g] (t',x',v') | \leq
C [ \|f_{0}\|_{C_{b}^{1}(\T^{2} \times \R^{2})} + 
\| (1+ |v|)^\gamma f_0 \|_{L^\infty(\T^{2} \times \R^{2})} ] \\
\times (1+|v|) |(t,x,v)- (t',x',v')|,
\end{multline}
and also
\begin{equation}
\label{EstGradV3}
| \tilde{{\mathcal V}}[g] (t,x,v) - \tilde{{\mathcal V}}[g] (t,x',v') | \leq
C [ \|f_{0}\|_{C_{b}^{1}(\T^{2} \times \R^{2})} + 
\| (1+ |v|)^\gamma f_0 \|_{L^\infty(\T^{2} \times \R^{2})} ]
|(x,v)- (x',v')|,
\end{equation}
the constant $C$ being independent of $f_0$.
\end{lem}
This lemma is rather technical. Actually without absorption, this estimate follows from Gronwall's lemma and the regularity of $\tilde{\mathcal V}[g]$ follows from the fact that $f_{0}$ and the characteristics are of class $C^{1}$. But here at each passage in $\gamma^{-}$, there is a jump between $\nabla \tilde{{\mathcal V}}[g](t^{+},x,v)$ and  $\nabla \tilde{{\mathcal V}}[g](t^{-},x,v)$. One can see by using an explicit computation based on the last equation in \eqref{EqLin3} that
\begin{equation*}
|\nabla \tilde{{\mathcal V}}[g](t^{+},x,v)| \leq  |\nabla \tilde{{\mathcal V}}[g](t^{-},x,v)| + C |\tilde{{\mathcal V}}[g](t^{-},x,v)|,
\end{equation*}
where $\nabla$ is either $\nabla_{x}$ or $\nabla_{v}$. \par
The main point is that the number $n(x,v)$ of times a characteristic $(X,V)(t,0,x,v)$ can cross $\gamma^{-}$ is estimated as follows. Using $\mbox{dist}(\gamma^{-},\gamma^{+})>0$ and Lemma \ref{LemCrucial}, we infer that
\begin{equation} \nonumber
n(x,v) \leq C (1+ \max_{t} |V(t,0,x,v)| ) \leq C(1 + |v|).
\end{equation}
This allows to bound $\nabla \tilde{{\mathcal V}}[g]$ using to the uniform estimates on $(1+ |v|)^{\gamma} \tilde{{\mathcal V}}[g]$. \par
\ \par
Finally, point {\bf a.} is a consequence of points {\bf b.}, {\bf c.} and an easy interpolation argument between weighted H\"older spaces, provided that $f_{0}$ is small enough. \par
\ \par
\subsection{A fixed point is relevant}
\label{Relevant}
Let us prove that, provided that $\varepsilon$ is small enough, the fixed point that we constructed is indeed a solution $f$ starting at $f_{0}$ and reaching $0$ in $\T^{2} \setminus \omega$ at time $T$. For this we show that $\tilde{V}[g](T)=0$ in $\T^{2} \times \R^{2}$. \par
Call again $(X,V)$ the characteristics associated to $F + \nabla \phi^{f} -\nabla \overline{\phi} + \overline{\mathcal E}$. \par
\ \par
\noindent
Due to the construction, it is enough to prove the following lemma. \par
\begin{lem} \label{LemCaracRencontre}
There exists $\epsilon_1>0$ such that for any $0<\epsilon<\epsilon_1$, all the characteristics $(X,V)$ meet $\gamma^{3-}$ for some time in $[\frac{T}{24},\frac{23T}{24}]$.
\end{lem}
\begin{proof}[Proof of Lemma \ref{LemCaracRencontre}]
We denote by $(\overline X, \overline V)$ the characteristics associated to $F + \overline{{\mathcal E}}$. 
%
\ \par
\noindent
{\bf 1.} We first prove that  for all $(x,v) \in \T^{2} \times \R^{2}$, there exists $\sigma \in [\frac{T}{12},\frac{3T}{12}] \cup [\frac{9T}{12},\frac{11T}{12}]$ such that
\begin{equation} \label{DansG4}
\overline{X}(\sigma,0,x,v) \in \gamma^{4-}:= \left\{ \II (x,v) \in S(x_0,r_0) \times
\R^{2} \ /\ |v| \geq \frac{5}{2} \text{ and } v.\nu(x) \leq -\frac{1}{4} |v| \right\}.
\end{equation}
Let $(x,v) \in \T^{2} \times \R^{2}$.
We claim that that there exists $t \in [\frac{T}{9},\frac{2T}{9}] \cup [\frac{7T}{9}, \frac{8T}{9}]$ such that
\begin{equation} \label{Danslaboule}
\overline{X}(t,0,x,v) \in B(x_{0},r_{0}/4),
\end{equation}
and
\begin{equation}
\label{Danslaboulebis}
|\overline{V}(t,0,x,v))| \geq \alpha.
\end{equation}
%
We discuss this according to the modulus of $V(T/3,0,x,v)$:
\begin{itemize}
\item If $|V(T/3,0,x,v)| \geq M \geq M_1$, then one can observe that $|v| \geq \max(\underline{m}_1,2\alpha)$, using the characteristics equation. Then by Proposition \ref{PropGV}, the claim is proved for some $t \in [\frac{T}{9},\frac{2T}{9}]  $.
\item If $|V(T/3,0,x,v)| <  M$, then by Proposition \ref{PropAccelerePartout}, $|V(2T/3,0,x,v)| \geq M+1 \geq M_2$, and one can once again apply Proposition \ref{PropGV}, to prove the claim for some $t \in [\frac{7T}{9}, \frac{8T}{9}]$.
\end{itemize}
Now, one can easily see that for some $s >0$ with $s < \frac{3 r_{0}}{\alpha} \leq \frac{T}{200}$,
\begin{equation}
\label{geometricfree}
\overline{X}(t,0,x,v) - s \overline{V}(t,0,x,v) \in S(x_{0},r_{0}) \text{ with } \overline{V}(t,0,x,v).\nu \leq - \frac{\sqrt{3}}{2} | \overline{V}(t,0,x,v)|,
\end{equation}
because a straight line arising from $B(x_{0},r_{0}/2)$ cuts $S(x_{0},r_{0})$ with angle to the normal $\nu$ at the circle of value at most $\pi/6$. The same argument shows that:
$$
\overline{X}(t,0,x,v) -2s \overline{V}(t,0,x,v) \notin B(x_{0},3r_{0}/2).
$$
Now it is clear that,
\begin{equation}
\label{estim1}
|\overline{V}(\tau,0,x,v) - \overline{V}(t,0,x,v) | \leq 2 s [\| F \|_{\infty} + \| \nabla \overline{\phi}_{1} \|_{\infty} + \| \nabla \overline{\phi}_{3} \|_{\infty}] \text{ for } \tau \in [t-2s,t],
\end{equation}
\begin{multline}
\label{estim2}
|\overline{X}(\tau,0,x,v) -\overline{X}(t,0,x,v) +(t-\tau) \overline{V}(t,0,x,v) | \\ 
\leq 2s^{2} [\| F \|_{\infty} + \| \nabla \overline{\phi}_{1} \|_{\infty}
+ \| \nabla \overline{\phi}_{3} \|_{\infty}] \text{ for } \tau \in [t-2s,t].
\end{multline}
In the other hand, if $C_{r_{0}}$ is large enough, we have the estimate:
\begin{equation*}
|\overline{X}(t-2s,0,x,v) -\overline{X}(t,0,x,v) +2s \overline{V}(t,0,x,v) | \leq \frac{r_{0}}{2}.
\end{equation*}
Therefore by the intermediate value theorem that there exists $\sigma \in [t- \frac{T}{100},t]$, such that $\overline{X}(\tau,0,x,v) \in S(x_{0},r_{0})$.
Using \eqref{geometricfree}, \eqref{estim1} and \eqref{estim2}, and provided that $C_{r_0}$ is large enough (in terms of $r_0$ only),  we deduce that for this $\sigma$, \eqref{DansG4} applies. %
\ \par
\ \par
\noindent
{\bf 2.} Now to prove that all the characteristics meet $\gamma^{3-}$ during $[\frac{T}{12},\frac{11T}{12}]$, let us compare $(\overline{X},\overline{V})$ and $(X,V)$. Using point {\bf a.} in the definition of ${\mathcal S}_{\varepsilon}$, we deduce by Gronwall's lemma and elliptic estimates that
\begin{equation*}
| (X,V) - (\overline{X},\overline{V}) | \leq C \varepsilon.
\end{equation*}
Proceeding as previously, we deduce that if $\varepsilon$ is small enough, then for all $(x,v) \in \T^{2} \times \R^{2}$, there exists $t \in [\frac{T}{24},\frac{23T}{24}]$, such that
\begin{equation*}
(X,V)(t,0,x,v) \in \gamma^{3-}.
\end{equation*}
\end{proof}
We can now gather all the ingredients to prove Theorem  \ref{Theo:Bounded}.
\begin{proof}[Proof of Theorem \ref{Theo:Bounded}]
Using Lemma \ref{Lem:ExistFP}, we deduce the existence of some fixed point $f={\mathcal V}[f]$. Using Lemma \ref{LemCaracRencontre}, and \eqref{DefU}, \eqref{DefUpsilon} and \eqref{EqLin3}, we see that, provided that $\varepsilon$ is small enough, $\tilde{{\mathcal V
}}[f](T)=0$. Hence $f$ satisfies $\mbox{Supp\,}[f(T,\cdot,\cdot)] \subset \omega \times \R^{2}$. \par
It remains to prove that $f$ satisfies \eqref{Vlasov}. This comes from the fact that, due to \eqref{EetPhi} and \eqref{EqLin3}, one has
\begin{equation*}
\partial_{t} f + v.\nabla_{x} f + (F+\nabla \phi^{f}).\nabla_{v} f = 0
\text{ in } [0,T] \times [\T^{n} \setminus \omega] \times \R^{n}.
\end{equation*}
Since $f$ is $C^{1}$, one has
\begin{equation*}
\partial_{t} f + v.\nabla_{x} f + (F+\nabla \phi^{f}).\nabla_{v} f = G
\text{ in } [0,T] \times \T^{n} \times \R^{n},
\end{equation*}
for some continuous function $G$. This concludes the proof of Theorem \ref{Theo:Bounded}.
\end{proof}
\section{Global controllability for the bounded external field case}
\label{Sec:BEFglobal}
In this section, we prove Theorem \ref{Theo:BoundedGlobal}. \par
\ \par
We call $H$ a hyperplane in $\R^{n}$ such that its image ${\mathcal H}$ by the canonical surjection $s:\R^{n} \rightarrow \T^{n}$ is included in $\omega$. We recall that ${\mathcal H}$ is supposed to be closed. We call $n_{H}$ a unit vector, orthogonal to ${\mathcal H}$. For $l>0$, we denote 
\begin{equation*}
{\mathcal H}_{l}:= {\mathcal H} + [-l,l] n_{H}.
\end{equation*}
Since ${\mathcal H}$ is closed in $\T^{n}$, we can define ${d} \in \R^{+*}$ such that 
\begin{equation*}
{\mathcal H}_{2d} \subset \omega,
\end{equation*}
and such that $4{d}$ is less than the distance between two different hyperplanes in $s^{-1}({\mathcal H})$. \par
%
%
%
%
%
%
\subsection{Design of the reference solution}
The reference solution is not quite the same as in Section \ref{Sec:BEF}. In order to get a global result, as explained in Section \ref{Sec:strategie}, we will need the following property, refered to as a ``non concentration property'' for the characteristics $(X,V)$ associated to $\overline{\varphi}$ (up to a slight modification {\it inside the control zone}): there exist $c >0$ such that
\begin{equation*}
\forall x,y \in \T^{n}, \ | X(t,0,x,0) - X(t,0,y,0)| \geq c |x-y|.
\end{equation*}

The assumption on the control zone $\omega$ is motivated by the fact that in this case we can atually construct a reference solution whose characteristics satisfy this condition. \par
\ \par
To construct $(\overline{\phi},\overline{f})$, we start with the following lemma.
\begin{lem}\label{LemPhi}
There exists $\varphi \in C^{\infty}(\T^{n};\R)$ such that
\begin{equation} \label{Harm}
\Delta \varphi =0 \text{ on } \T^{n} \setminus {\mathcal H}_{d},
\end{equation}
and
\begin{equation} \label{CoincideavecN}
\nabla \varphi = n_{H} \text{ on } \T^{n} \setminus {\mathcal H}_{d}.
\end{equation}
\end{lem}
\begin{proof}[Proof of Lemma \ref{LemPhi}]
In the domain $\T^n \backslash {\mathcal H}_{d}$, $x \mapsto n_{H}$
coincides with the gradient of a harmonic function. Call $\varphi$ a
function in $C^\infty(\T^n;\R)$, whose gradient coincides in ${\mathcal H}_{d}$ with $n_{H}$; 
this function is automatically harmonic in ${\mathcal H}_{d}$. 
\end{proof}
Now given such a $\varphi$, we can construct $\overline{\phi}$ and $\overline{f}$.
Consider a function ${\mathcal Y} \in C^\infty_0(0,T)$
satisfying
\begin{equation}
\label{DefY}
\left\{ \begin{array}{l}
{\mbox{Supp } {\mathcal Y} \subset (\frac{T}{3},\frac{2T}{3}),} \\
{ {\mathcal Y} \geq 0,} \\
{ \displaystyle \int_{[0,T]} {\mathcal Y} =1.}
\end{array} \right.
\end{equation}
Set
\begin{equation*}
\overline{\phi}(t,\cdot) = \left\{ \begin{array}{l}
 \displaystyle 0 \text{ for } t \in \left[0,\frac{T}{3}\right] \cup \left[\frac{2T}{3},T\right], \\
 \displaystyle \mu {\mathcal Y} (t) \varphi(\cdot) \text{ for } t \in \left[\frac{T}{3},\frac{2T}{3}\right],
\end{array} \right.
\end{equation*}
\begin{equation*}
\overline{{\mathcal E}}(t,\cdot) = \left\{ \begin{array}{l}
 \displaystyle 0 \text{ for } t \in \left[0,\frac{T}{3}\right] \cup \left[\frac{2T}{3},T\right], \\
 \displaystyle \mu {\mathcal Y} (t) n_{H} \text{ for } t \in \left[\frac{T}{3},\frac{2T}{3}\right],
\end{array} \right.
\end{equation*}
where $\mu$ is a positive parameter depending on $\omega$, $T$ and $F$ only, according to the following lemma.
\begin{lem}\label{LemExistMu}
Given $\omega$ as above, $T>0$ and $F$, there exists $\mu>0$ such that all the characteristics associated to $\overline{{\mathcal E}}$ meet 
\begin{equation} \label{DefGammaMoinsMoins2} 
\gamma^{3-}:=\left\{ (x,v) \in  \partial {\mathcal H}_{d} \times \R^{n} \ /\ |v| \geq 2 \text{ and } v.\nu \leq - 2 \right\},
\end{equation}
for some time in $[\frac{T}{6},\frac{5T}{6}]$, where $\nu = \pm n_{H}$ is the outward unit vector on $\partial {\mathcal H}_{d}$.
\end{lem}
Once defined $\overline{\varphi}$, we define $\overline{f}:[0,T] \times \T^{2} \times \R^{2}$ as previously by \eqref{DefZ}-\eqref{Deffbar}.
\begin{proof}[Proof of Lemma \ref{LemExistMu}]
Let $(x,v) \in \T^{n} \times \R^{n}$. Call $(\overline{X},\overline{V})$ the characteristics associated to $\overline{{\mathcal E}}$.  We discuss according to $V(\frac{T}{6},0,x,v) \cdot n_{H}$.
\begin{itemize}
\item If $V(\frac{T}{6},0,x,v) \cdot n_{H}$ is large enough, say larger than $c>0$, then one sees easily using the characteristic equation that there exists $t \in [\frac{T}{6},\frac{T}{4}]$ such that $(\overline{X},\overline{V})(t,0,x,v) \in \gamma^{3-}$.
\item For the other $(x,v)$, one can find $\mu>0$ such that $V(\frac{2T}{3},0,x,v) \cdot n_{H} \geq c$. Then there exists  $t \in [\frac{2T}{3},\frac{5T}{6}]$ such that $(\overline{X},\overline{V})(t,0,x,v) \in \gamma^{3-}$.
\end{itemize}
\end{proof}
\subsection{Definition of the fixed-point operator}
For $\lambda \in (0,1]$, we define again a subset ${\mathcal S}^{\lambda}_{\varepsilon}$ of
$C_{b}^{\delta_2}(Q_{T})$ on which we will define the operator
${\mathcal V}$ (which actually depends on $\lambda$): 
\begin{gather}
\nonumber
\begin{array}{ll}
{{\mathcal S}^{\lambda}_{\varepsilon}:= 
\left\{ \II \ g \in C_{b}^{\delta_2}(Q_{T}) \ \right/ \hfill} & {\hfill }
\end{array} \\
\label{DefS2}
\begin{array}{ll}
{\mathbf a.\ }  &
{\|  \int_{\mathbb{R}^n} (g - \overline{f})dv \|_{C^{\delta_1}(\Omega_{T})} 
\leq \varepsilon,} \\
{\mathbf b.\ } &
{\left. \| (1+ |v|)^{\gamma}( g - \overline{f})\|_{L^{\infty}(Q_{T})} 
\leq c_{1} 
\left[ \II \|f^{\lambda}_{0}\|_{C_{b}^{1}(\T^{n} \times \R^{n})}
+\|(1+|v|)^{\gamma}f^{\lambda}_{0}\|_{C_{b}^{0}(\T^{n} \times \R^{n})} \right], 
\right. } \\
{\mathbf c.\ } &
{\left. \| g - \overline{f} \|_{C_{b}^{\delta_2}(Q_{T})} \leq c_2
\left[ \II \|f^{\lambda}_{0}\|_{C_{b}^{1}(\T^{n} \times \R^{n})}
+\|(1+|v|)^{\gamma}f^{\lambda}_{0}\|_{C_{b}^{0}(\T^{n} \times \R^{n})} \right] ,
\right.} \\
{\mathbf d.\ } &
{\left. \II \forall t \in [0,T],\ \int_{\T^{n} \times \R^{n}}
g(t,x,v) dx dv = \int_{\T^{n} \times \R^{n}} f^{\lambda}_{0}(x,v) 
dx dv\ \right\} ,}
\end{array}
\end{gather}
with $c_1$, $c_2$ to be fixed later depending only on $\gamma$, $T$
and $\omega$ (and hence on $(\overline{f},\varphi)$), but not on
$\lambda$; here, $\delta_1$ and $\delta_2$ are fixed as follows
\begin{equation}
\nonumber
\delta_1= \frac{\gamma -n}{2(n+1)(\gamma+1)} \text{ and } \delta_2=\frac{\gamma}{\gamma+1}.
\end{equation}
For fixed $c_1$ and $c_2$ large enough depending only on
$(\overline{f},\varphi)$, one has for $\lambda$ small enough depending on
$\varepsilon$ that 
\begin{equation*}
\left|\int f_0^\lambda  dvdx \right| \leq \varepsilon,
\end{equation*}
see (\ref{ChgtEchelle}). Hence in that case $g(t,x,v)=f^\lambda_0(x,v)+\overline{f}(t,x,v)$ belongs to
${\mathcal S}^\lambda_\varepsilon$ for $\lambda < \mu(\varepsilon)$,
so  ${\mathcal S}^\lambda_\varepsilon
\not = \emptyset$. From now, we suppose that this is the
case. \par
We write $\Gamma_{1}:={\mathcal H} -d n_{H}$, $\Gamma_{2}:={\mathcal H} +d n_{H}$  and $\Gamma:=\Gamma_{1} \cup \Gamma_{2}$. Let $\nu = -n_{h}$ on $\Gamma_{1}$ and $\nu = n_{h}$ on $\Gamma_{2}$. We define
\begin{gather}
\label{DefGammaMoins2} 
\gamma^{-}:= \left\{ \II (x,v) \in \Gamma \times
\R^{n} \ / \ v.\nu(x) < -1 \right\}, \\
\label{DefGamma32Moins2} 
\gamma^{2-}:= \left\{ \II (x,v) \in \Gamma \times
\R^{n} \ /\ |v| \geq 1 \text{ and } v.\nu(x) \leq - 3/2 \right\}, \\
\label{DefGammaPlus2} 
\gamma^{+}:= \left\{ \II (x,v) \in \Gamma \times
\R^{n} \ /\ v.\nu(x) \geq 0 \right\}.
\end{gather}
Note that $\gamma^{3-}$ defined in \eqref{DefGammaMoinsMoins2} can be reformulated as
\begin{equation*}
\gamma^{3-}= \left\{ \II (x,v) \in  \Gamma \times
\R^{n} \ /\ |v| \geq 2 \text{ and } v.\nu(x) \leq - 2 \right\}.
\end{equation*}
%
%
\ \par
Again, we observe that
\begin{equation}
\nonumber
\mbox{dist}( (\Gamma \times \R^n) \backslash \gamma^{-}; \gamma^{2-}) >0.
\end{equation}
\ \\
We introduce a $C^{\infty} \cap C_{b}^{1}$ regular function $U$ from
$\Gamma \times \R^{n}$ to $\R$ the same way as previously, by
\begin{equation}
\label{DefU2}
\left\{ \begin{array}{l}
{ 0 \leq U \leq 1, } \\
{U \equiv 1 \text{ in } (\Gamma \times \R^n) \backslash \gamma^{-},} \\
{U \equiv 0 \text{ in } \gamma^{2-}.}
\end{array} \right.
\end{equation}
The function $\Upsilon$ is again introduced by (\ref{DefUpsilon}). As in Section \ref{Sec:BEF},
we define $\pi$ as a continuous affine extension
operator $\overline{\pi}$ from $C^{0}({\mathcal H}_{2d};\R)$ to $C^{0}(\T^{n};\R)$, and which has the same property that
each $C^{\alpha}$-regular function is continuously mapped to a
$C^{\alpha}$-regular function, for any $\alpha \in [0,1]$. Moreover, we manage again in order that for any $f \in C^{0}({\mathcal H}_{2d};\R)$, (\ref{NeutralitedePi}) occurs. The operator $\Pi$ is given by (\ref{RegleGrandPi}). \par
\ \par
Now, given $g \in {\mathcal S}^{\lambda}_{\varepsilon}$, we first define $\phi^g$ by \eqref{defphiPoisson}. \par
%
%
%
Then we introduce $f=\tilde{{\mathcal V}}[g]$ as the solution of the following system:
\begin{equation}
\label{EqLin3bis}
\left\{ \begin{array}{l}
f(0,x,v) = f^{\lambda}_{0} \text{ on } \T^{n} \times \R^{n}, \medskip \\
\partial_{t} f + v.\nabla_{x} f + (F^{\lambda}+\nabla (\phi^{g} - \overline{\varphi}) + \mu {\mathcal Y}(t) n_{H} ).\nabla_{v} f = 0 \text{ in } [0,T] \times [(\T^{n} \times \R^{n}) \backslash \gamma^{-}], \medskip \\
f(t,x,v) = [ 1- \Upsilon(t)] f(t^{-},x,v) + \Upsilon (t) U(x,v) f(t^{-},x,v) \text{ on } [0,T] \times \gamma^{-}.
\end{array} \right.
\end{equation} 
The meaning of this equation is the same one as in Section \ref{Sec:BEF} (and $\mu {\mathcal Y}(t) n_{H}$ plays the same role as ${\mathcal E}$ in Section \ref{Sec:BEF}). Recall that $F^{\lambda}$ was defined in \eqref{ChgtEchelleF}. \par
Then, as for Section \ref{Sec:BEF}, we define
${\mathcal V}[g]$ by
\begin{equation}
\label{DefV2b}
{\mathcal V}[g]:= \overline{f} + \Pi ( f_{|[0,T] \times {\mathcal H}_{2d}
\times \R^{n} \cup [0,T/48] \times \T^n \times \R^n } ) \text{ in } [0,T] \times \T^{n} \times \R^{n}.
\end{equation}
Again, $f_{|[0,T] \times {\mathcal H}_{2d} \times \R^{n} \cup [0,T/48] \times \T^n \times \R^n }$ is $C^{1}$
regular, and, together with the construction of $\Pi$, it will follow
that ${\mathcal V}[g]$ is in $C^{1}([0,T] \times \T^{n} \times \R^{n})$.  \par
\ \par
Considering the form of \eqref{EqLin3bis}, the characteristics that we consider in the sequel are $(X^{g},V^{g})$ associated to $F^{\lambda}+\nabla (\phi^{g} - \overline{\varphi}) + \mu {\mathcal Y}(t) n_{H}$, which coincide with the ones associated to $F^{\lambda}+\nabla \phi^{g}$ outside the control zone, but not necessarily inside. \par
\subsection{Existence of a fixed point}
Now our goal is to prove the following lemma. 
\begin{lem}\label{LemFPbis}
For any small $\varepsilon>0$, there exists $\overline{\lambda}(\varepsilon)>0$ such that for any positive $\lambda <
\overline{\lambda}(\varepsilon)$, the operator ${\mathcal V}$ has a fixed point in ${\mathcal S}^{\lambda}_{\varepsilon}$.
\end{lem}
\begin{proof}[Proof of Lemma \ref{LemFPbis}]
We prove Lemma \ref{LemFPbis} by checking the assumptions for Schauder's fixed point Theorem on ${\mathcal V}$. We will sometimes forget the
indices and exponents $\varepsilon$ and $\lambda$. \par
\ \\
{\bf 1.} Again, ${\mathcal S}$ is a convex compact subset of
$C^{0}(Q_{T})$. \par
\ \\
{\bf 2.} The continuity of ${\mathcal V}$ can be proven in the same way as in
Section 3. \par
%
%
\ \\
%
%
%
%
{\bf 3.} The difficulty is to check that for $\lambda$ small, one has ${\mathcal V}({\mathcal S}^{\lambda}_{\varepsilon}) \subset {\mathcal S}^{\lambda}_{\varepsilon}$. Accordingly, we have to check the points {\bf a.}, {\bf b.}, {\bf c.} and {\bf d.} for ${\mathcal V}[g]$. \par
That ${\mathcal V}[g]$ satisfies {\bf d.} comes directly from the construction. That $\tilde{{\mathcal V}}[g]$ and consequently ${\mathcal V}[g]$ satisfies estimates as {\bf b.} is not difficult and proven as in Section \ref{Sec:BEF}. In particular Lemma \ref{LemCrucial} is still satisfied. \par
For what concerns point {\bf c.} we have as previously (see also \cite[Lemma 4, p. 370]{OG03})

\begin{lem}
\label{LemmeRegHV2}
For $g \in {\mathcal S}^{\lambda}_\varepsilon$, one has $\tilde{{\mathcal V}}[g] \in
C^{1}(Q_{T} \backslash \Sigma_{T})$, with $\Sigma_{T}:= [0,T] \times \gamma^{-}$.
Moreover, for any $(t,x,v)$ and $(t',x',v')$ in $[0,T]
\times [ \T^{2} \backslash \omega ] \times \R^{2}$, with $|
v-v' | \leq 1$, one has,
\begin{multline}
\label{EstGradV2b}
| \tilde{{\mathcal V}}[g] (t,x,v) - \tilde{{\mathcal V}}[g] (t',x',v') | \leq
C [ \|f_{0}\|_{C_{b}^{1}(\T^{2} \times \R^{2})} + 
\| (1+ |v|)^{\gamma+2} f_0 \|_{L^\infty(\T^{2} \times \R^{2})} ] \\
\times (1+|v|) |(t,x,v)- (t',x',v')|,
\end{multline}
and also
\begin{equation}
\label{EstGradV3b}
| \tilde{{\mathcal V}}[g] (t,x,v) - \tilde{{\mathcal V}}[g] (t,x',v') | \leq
C [ \|f_{0}\|_{C_{b}^{1}(\T^{2} \times \R^{2})} + 
\| (1+ |v|)^{\gamma+2} f_0 \|_{L^\infty(\T^{2} \times \R^{2})} ]
|(x,v)- (x',v')|,
\end{equation}
the constant $C$ being independent from $f_0$.
\end{lem}
The central part is point {\bf a.}, where the smallness of $\lambda$ and the non concentration property of $\overline{\varphi}$ are used.
We begin by a lemma which asserts that the non concentration property is preserved by a small perturbation. Recall that $(X^{g},V^{g})$ are associated to $F^{\lambda}+\nabla (\phi^{g} - \overline{\varphi}) + \mu {\mathcal Y}(t) n_{H}$.
\begin{lem}\label{LemNonFocalization}
There exists $c>0$ such that for any $\lambda$ small enough (in terms of $T$, $\omega$ and $F$),
for any $g \in {\mathcal S}^{\lambda}_{\varepsilon}$, one has
\begin{equation}
\label{XgNonMelangeante}
\forall (x,y) \in (\T^{n} )^2 ,\ \forall t \in [0,T],\ \ c^{-1} |x-y|
\leq |X^{g}(t,0,x,0) - X^{g}(t,0,y,0)|
\leq c |x-y|.
\end{equation}
\end{lem}
\begin{proof}[Proof of Lemma \ref{LemNonFocalization}]
Define $(\overline{X},\overline{V})$ as the characteristics associated to the force $\mu {\mathcal Y}(t) n_{H}$. It is clear that $(\overline{X},\overline{V})$ satisfy the non concentration property:
\begin{equation}
\label{XVBarNonFocalisants}
\forall (x,y) \in (\T^{n} )^2 ,\ \forall t \in [0,T],\ \
 |\overline{X}(t,0,x,0) - \overline{X}(t,0,y,0)|
\geq  |x-y|.
\end{equation}
(This is actually an equality!)
Now, it follows from Gronwall's inequality that for a constant $C$ depending only on $\mu$, ${\mathcal Y}$ and $F$, one has
\begin{equation}
\label{estGronwallglobal}
\| (X^{g},V^{g}) - (\overline{X},\overline{V})\|_{C_{b}^{0}([0,T]^2 \times \T^n \times \R^n)} 
\leq C (\varepsilon + \lambda^{2}).
\end{equation}
One can get a further inequality in the following way (when it is
not explicit, the norm considered is the $L^\infty$ one)
\begin{align*}
\nonumber
\frac{d}{dt^+} & \| \nabla (X^{g},V^{g})(t,s,x,v)  - \nabla (\overline{X},\overline{V})(t,s,x,v) \| \\
& \leq \| \nabla V^{g}(t,s,x,v) - \nabla \overline{V}(t,s,x,v) \| \\
& +\| \nabla_x E_g(t,X^{g}(t,s,x,v)) \nabla X^{g}(t,s,x,v) -
\nabla_x E_{\overline{f}}(t,\overline{X}(t,s,x,v)) 
\nabla \overline{X}(t,s,x,v)\| \\
& + \| \nabla_{x,v} F^{\lambda}(t,X^{g}(t,s,x,v), V^{g}(t,s,x,v)) \nabla (X^{g},V^{g})(t,s,x,v) \\
& \hskip 1cm - \nabla_{x,v} F^{\lambda}(t,\overline{X}(t,s,x,v), \overline{V}(t,s,x,v)) \nabla (\overline{X},\overline{V})(t,s,x,v) \|
\end{align*}
where $\nabla$ stands either for $\nabla_x$ or for $\nabla_v$. Now
the second term is bounded as follows
\begin{equation}
\nonumber
\| \nabla_x E_g(t,X^{g}(t,s,x,v)) \nabla X^{g}(t,s,x,v) -
\nabla_x E_{\overline{f}}(t,\overline{X}(t,s,x,v)) \nabla
\overline{X}(t,s,x,v)\| \leq  A_{1}+A_{2}+A_{3},
\end{equation}
with
\begin{eqnarray}
\nonumber
\left\{ \begin{array}{l}
{A_{1} = \| \nabla_x E_g(t,X^{g}(t,s,x,v)) \nabla X^{g}(t,s,x,v)
- \nabla_x E_{g}(t,X^g(t,s,x,v)) \nabla
\overline{X}(t,s,x,v)\|,} \\
{A_{2}= \| \nabla_x E_{g}(t,X^{g}(t,s,x,v)) \nabla
\overline{X}(t,s,x,v) - \nabla_x E_{\overline{f}}(t,X^g(t,s,x,v))
\nabla \overline{X}(t,s,x,v)\|,} \\
{A_{3}=\| \nabla_x E_{\overline{f}}(t,X^{g}(t,s,x,v)) \nabla
\overline{X}(t,s,x,v) - \nabla_x
E_{\overline{f}}(t,\overline{X}(t,s,x,v)) \nabla
\overline{X}(t,s,x,v)\|.}
\end{array} \right.
\end{eqnarray}
Now
\begin{eqnarray}
\nonumber
A_{1} &\leq& \|\nabla_x E_g\|_{C_{b}^0(\Omega_T)} \| \nabla
X^{g}(t,s,x,v) - \nabla \overline{X}(t,s,x,v) \|_{C_{b}^0([0,T]^2 \times \T^n \times \R^n)}, \\
\nonumber
A_{2} &\leq& \| \nabla_x E_{g} - \nabla_x
E_{\overline{f}}\|_{C_{b}^0(\Omega_T)} \| \nabla
\overline{X}\|_{C_{b}^0([0,T]^2 \times \T^n \times \R^n)}, \\ 
\nonumber
A_{3} &=&0.
\end{eqnarray}
Hence we obtain
\begin{multline*}
\| \nabla_x E_g(t,X^{g}(t,s,x,v)) \nabla X^{g}(t,s,x,v) -
\nabla_x E_{\overline{f}}(t,\overline{X}(t,s,x,v)) \nabla
\overline{X}(t,s,x,v)\| \\ \leq C ( \varepsilon + \| \nabla
X^{g}(t,s,x,v) - \nabla \overline{X}(t,s,x,v) \|_{C_{b}^0([0,T]^2 \times \T^n \times \R^n)}).
\end{multline*}

We treat the term concerning $F^{\lambda}$ in the same way and obtain
\begin{multline*}
\| \nabla_{x,v} F^{\lambda}(t,X^{g}(t,s,x,v), V^{g}(t,s,x,v)) \nabla (X^{g},V^{g})(t,s,x,v) \\
 - \nabla_{x,v} F^{\lambda}(t,\overline{X}(t,s,x,v), \overline{V}(t,s,x,v)) \nabla (\overline{X},\overline{V})(t,s,x,v) \| \\
\leq C (\lambda 
+ \| \nabla (X^{g},V^{g})(t,s,x,v) - \nabla (\overline{X},\overline{V})(t,s,x,v) \|_{C_{b}^0([0,T]^2 \times \T^n \times \R^n)}).
\end{multline*}
It follows then by Gronwall's lemma that for a certain constant $C$,
one has
\begin{equation}
\nonumber
\| (X^{g},V^{g}) - (\overline{X},\overline{V})\|_{L^\infty([0,T];
C_{b}^{1}(\T^n \times \R^n))} \leq C (\varepsilon + \lambda).
\end{equation}
Hence, if $\varepsilon$ and $\lambda$ are small enough, then \eqref{XVBarNonFocalisants} is still valid when replacing
$(\overline{X},\overline{V})$ by $(X^{g},V^{g})$, up to a multiplicative constant. This gives \eqref{XgNonMelangeante}.
\end{proof}

Let us come back to the proof of point {\bf a}. Let us treat the $L^{\infty}$-norm; the $C^{\delta_1}$ one will follow by interpolation.
From (\ref{XgNonMelangeante}), we deduce that
$X^{g}(t,0,\cdot,0):\T^{n} \rightarrow \T^{n}$ is invertible; call
$({X}_{t}^{g})^{-1}$ its inverse, and define the function ${W}_{t}^{g}: [0,T] \times \T^{n} \rightarrow \R^{n}$ by
\begin{equation}
\nonumber
W_{t}^{g}(x):= V^{g}(t,0,({X}_{t}^{g})^{-1}(x),0).
\end{equation} \par
One can describe $({X}_{t}^{g})^{-1}(x)$ as the initial position of a particle, which starting with velocity $0$, reaches $x$ at time $t$; then $W_{t}^{g}(x)$ is its velocity at time $t$. \par
Let us give an estimate on $v - W_{t}^{g}(x)$. 
First,
\begin{equation*}
v- W^{g}_{t}(x) = V^{g}(0,t, X^{g}(t,0,x,v),V^{g}(t,0,x,v)) - V^{g}(t,0,({X}_{t}^{g})^{-1}(x),0).
\end{equation*}
By using Gronwall's lemma on $V(0,t,\cdot,\cdot)$, we deduce that for some constant independent of $\lambda \in (0,1]$
\begin{equation*}
|v- W^{g}_{t}(x)| \leq C\left( | X^{g}(t,0,x,v) - ({X}_{t}^{g})^{-1}(x)| + |V^{g}(t,0,x,v)|\right).
\end{equation*}
That the constant is independent of $\lambda$ comes from the fact that we have uniform Lipschitz estimates on $F^{\lambda}+\nabla (\phi^{g} - \overline{\varphi}) + \mu {\mathcal Y}(t) n_{H}$ for $\lambda \in (0,1]$. \par
To estimate the first term, we first notice that the non-concentration property \eqref{XgNonMelangeante} gives
\begin{eqnarray*}
(c')^{-1} | (X^{g}_{t})^{-1}(x) - X^{g}(0,t,x,v) | 
&\leq& | X^{g}(t,0,(X^{g}_{t})^{-1}(x),0) - X^{g}(t,0,X^{g}(0,t,x,v),0) | \\
&=& |x -X^{g}(t,0,X^{g}(0,t,x,v),0) | \\
&=& |X^{g}(t,0,X^{g}(0,t,x,v),V^{g}(0,t,x,v))  -X^{g}(t,0,X^{g}(0,t,x,v),0) |
\end{eqnarray*}
where the first equality comes from the definition of $(X_{t}^{g})^{-1}$, and the second one of the flow property. \par
Now Gronwall's lemma for $X^{g}(t,0,\cdot,\cdot)$, we deduce that for some constant $C>0$ independent of $\lambda \in (0,1]$ one has
\begin{equation*}
|X^{g}(t,0,X^{g}(0,t,x,v),V^{g}(0,t,x,v))  -X^{g}(t,0,X^{g}(0,t,x,v),0) | \leq C |V^{g}(0,t,x,v)|.
\end{equation*}
Finally we deduce that for some constant $K > 0$ independent of $\lambda$, one has, for any
$\lambda \in (0,1]$ and any $g \in {\mathcal S}^{\lambda}_{\varepsilon}$,
\begin{equation} \label{PropagationVitesses}
|v- W^{g}_{t}(x)| \leq K|V^{g}(0,t,x,v)|.
\end{equation}
Now, one has
\begin{eqnarray}
\nonumber 
|f(t,x,v)| &\leq& | f_{0}^{\lambda} \left[
(X^{g},V^{g})(0,t,x,v) \right] | \\
\nonumber
& \leq& \lambda^{2-n} 
\|f_{0} (1+|v|)^{\gamma}\|_{L^{\infty}(\T^{n} \times \R^{n})}
\left( 1+\frac{1}{\lambda} \left| V^{g}(0,t,x,v) \right| \right)^{-\gamma}.
\end{eqnarray} 
Using (\ref{PropagationVitesses}), we get that
\begin{equation}
\nonumber 
|f(t,x,v)| \leq \lambda^{2-n} \|f_{0}
(1+|v|)^{\gamma}\|_{L^{\infty}(\T^{n} \times \R^{n})}
\left( 1+\frac{1}{K \lambda} |v - W_{t}(x)|
\right)^{-\gamma}.
\end{equation}
It follows that
\begin{equation}
\nonumber 
|\int_{\R^{n}} f(t,x,v) dv| \leq 
\lambda^{2-n} \|f_{0} (1+|v|)^{\gamma}\|
_{L^{\infty}(\T^{n} \times \R^{n})}
\int_{\R^{n}} \left(
1+\frac{1}{K \lambda} |v - W_{t}^{g}(x)| \right)^{-\gamma} dv.
\end{equation}
We deduce that
\begin{equation}
\nonumber
|\int_{\R^{n}} \tilde{{\mathcal V}}[g] (t,x,v) dv | \leq \kappa
\lambda^{2-n} \|f_{0} (1+|v|)^{\gamma}\|_{L^{\infty}(\T^{n} \times
\R^{n})} K^n \lambda^{n}.
\end{equation}
One deduces from the construction of ${\mathcal V}$ that
\begin{equation}
\label{EstDensiteLinfini2}
\| \int ({\mathcal V}[g]-\overline{f})(t,x,v) dv \|_{L^{\infty}(\Omega_{T})}
\leq C \lambda^{2-n} \|f_{0}(1+|v|)^{\gamma}\|_{L^{\infty}(\T^{n}
\times \R^{n})} \lambda^{n} \leq C(f_0) \lambda^{2}.
\end{equation}
Now we turn to the H\"older estimate. It follows by interpolation
between points {\bf b} and {\bf c}, that for a certain constant $C$
independent from $\lambda$, and for
$\tilde{\gamma}=\frac{n+\gamma}{2}$ and $\delta = \gamma/(\gamma+1)$
one has
\begin{equation}
\nonumber
| {\mathcal V}[g] - \overline{f}|^{\tilde{\gamma}}_{\delta} \leq C \left[ \II
\|f^{\lambda}_{0}\|_{C_{b}^{1}(\T^{n} \times \R^{n})}
+\|(1+|v|)^{\gamma}f^{\lambda}_{0}\|_{C^{0}(\T^{n} \times \R^{n})}
\right].
\end{equation}
We deduce that, for $\lambda \leq 1$ and another constant $C$
(depending on $f_0$ but not on $\lambda$),
\begin{equation}
\nonumber 
\| \int ({\mathcal V}[g] - \overline{f}) dv
\|_{C^{\delta}(\Omega_{T})} \leq C \lambda^{1-n}.
\end{equation}
Now we interpolate again this inequality with (\ref{EstDensiteLinfini2}).
We get that for $\delta_1$, one has
\begin{equation}
\nonumber 
\| \int  ({\mathcal V}[g] - \overline{f}) dv
\|_{C^{\delta_1}(\Omega_{T})} \leq K' \lambda.
\end{equation}
which concludes the point {\bf a}, for it is sufficient to find a
proper $\lambda$.
This finally proves ${\mathcal V}({\mathcal S}^{\lambda}_{\varepsilon}) \subset {\mathcal S}^{\lambda}_{\varepsilon}$. \par
\end{proof}
\subsection{A fixed point is relevant}
Now we can prove that the characteristics associated to the fixed point are relevant:
\begin{lem} \label{LemCaracRencontreGlobal}
There exists $\epsilon_1>0$ such that for any $0<\epsilon<\epsilon_1$, all the characteristics $(X,V)$ meet $\gamma^{2-}$ for some time in $[\frac{T}{24},\frac{23T}{24}]$.
\end{lem}
\begin{proof}[Proof of Lemma \ref{LemCaracRencontreGlobal}]
We recall that by the scaling $F^\lambda=\lambda^2 F(\lambda t, x, \frac{v}{\lambda})$, so that $\| F^\lambda \|_{L^\infty_{t,x,v}} \leq \lambda^2 \| F \|_{L^\infty_{t,x,v}}$. As for Lemma \ref{LemCaracRencontre}, the proof follows, recalling the Gronwall's estimate \eqref{estGronwallglobal}, and the fact that the characteristics associated to the reference solution $\overline{f}$ meet $\gamma^{3-} \times [\frac{T}{6},\frac{5T}{6}]$. (We recall that $\mu$ was defined when we have constructed the reference solution $\overline{f}$.)
\end{proof}
Finally, we can conclude the proof of the theorem.
\begin{proof}[Proof of Theorem \ref{Theo:BoundedGlobal}]
Using Lemma \ref{LemFPbis}, we deduce the existence of some fixed point $g$ for $\lambda$ sufficiently small. Using Lemma \ref{LemCaracRencontreGlobal}, and  \eqref{DefUpsilon}, \eqref{DefU2} and \eqref{EqLin3bis}, we see that it satisfies $\mbox{Supp\,}[g(T,\cdot,\cdot)] \subset \omega \times \R^{2}$ . Now, we define $f(t,x,v)=g(\frac{t}{\lambda},x,\lambda v)$, which satisfies the conclusions of Theorem \ref{Theo:BoundedGlobal}. The fact that \eqref{Vlasov} is satisfied for some $G$ supported in $\omega$ is done as in Section \ref{Sec:BEF}.
\end{proof}
\section{External magnetic field case}
\label{Sec:CM}
In this section, we prove Theorem \ref{Theo:Mag}, that is the local controllability result for the external magnetic field case.
\subsection{Rephrasing the geometric assumption}
We begin by transforming the geometric assumption \eqref{GeometricCondition} in a way that is easier to handle in the sequel. 
For $K$ a compact subset of $\T^{2}$ and $r>0$ we denote
\begin{equation} \label{Kr}
K_{r} := \{ x \in \T^{2} \ / \ d(x,K) \leq r \}.
\end{equation}

The geometric assumption can be reinterpreted with the help of the folllowing lemma.
\begin{lem} \label{LemEpaissi}
Let $K \subset \T^{2}$ such that $b >0$ on $K$ and satisfying \eqref{GeometricCondition}.
Then there exists $\underline{b}>0$, $d>0$ and $D>0$ such that
\begin{equation} \label{Petitd}
b \geq \underline{b} \ \text{ on } K_{2d},
\end{equation}
\begin{equation} \label{GrandD}
\forall x \in \T^{2}, \ \forall e \in \S^{1}, \ \exists t \in [0,D],
\ \forall s \in \left[t,t+ \frac{d}{2}\right], \  x + se \in K_{d}.
\end{equation}
\end{lem}
\begin{proof}[Proof of Lemma \ref{LemEpaissi}]
An easy argument relying on the compactness of $K$ shows that for $d>0$ suitably small, one has \eqref{Petitd}. \par
To prove \eqref{GrandD}, we use the compactness of $\T^{2} \times \S^{1}$. For any $(x,e) \in \T^{2} \times \S^{1}$, there exists $t \in \R^{+}$ such that $x+te \in K$. One deduces that for $(x',e')$ in an open neighborhood of $(x,e)$ in $\T^{2} \times \S^{1}$, one has $x'+te' \in K_{d/2}$. \par
Hence by compactness of $\T^{2} \times \S^{1}$, there exists a maximal time $D$ such that for any $(x,e) \in \T^{2} \times \S^{1}$, there exists $t \in [0,D]$ for which $x+te \in K_{d/2}$. Now if $x+te \in K_{d/2}$ and $x+t'e \notin K_{d}$, then one has $|t-t'| \geq d/2$, since $\mbox{dist} (K_{d/2}, \T^{2} \setminus K_{d}) \geq d/2$. The conclusion \eqref{GrandD} follows.
\end{proof}
\ \par
\subsection{Design of the reference solution}
The first step consists in building the reference solution, once again distinguishing between high and low velocities.
We first treat the case of large velocities. We prove that with the geometric assumption on $b$, high velocity particles spontaneously reach the arbitrary open set. One can observe that this is very different to the case of bounded force fields.
Actually we can prove a stronger result than announced, since we can add to the Lorentz force any additional bounded force field. Such a generalization will be actually crucial for the proof of Lemma \ref{LemCaracRencontre2}.
\begin{prop} \label{PropSolRefCM}
Let $T>0$ and $r_{0}>0$. Let $b$ satisfy the geometric condition \eqref{GeometricCondition}. There exists $\underline{m} \in \mathbb{R}^{+*}$ large enough depending only on $b$, $T$ and $\omega$ such that for all $\mathfrak{F} \in L^{\infty}(0,T;W^{1,\infty}(\T^{2} \times \R^{2}))$ satisfying $\| \mathfrak{F} \|_{L^{\infty}} \leq 1$, the characteristics $(\overline{X}, \overline{V})$ associated to $b(x) v^\perp +\mathfrak{F}$ satisfy:
\begin{multline} \label{GVCM}
\forall x \in \mathbb{T}^2, \forall v \in \mathbb{R} ^2 \text{ such that } \vert v \vert \geq \underline{m},
\exists t \in (T/4,3T/4), \ \overline{X}(t,0,x,v) \in B(x_0,r_0/2) \\
 \text{ and for all } s \in [0,T], \ \ \frac{|v|}{2} \leq |\overline{V}(s,0,x,v)| \leq 2|v|.
\end{multline}
\end{prop}
\begin{proof}[Proof of Proposition \ref{PropSolRefCM}]
We prove Proposition \ref{PropSolRefCM} in several cases of increasing complexity. In a first time (Cases 1--3), we suppose that $\mathfrak{F}=0$. In Case 4, we explain how to take $\mathfrak{F}$ into account. \par
In all cases, we define
\begin{equation} \label{DefOvelineEta}
\overline{b}:= \max_{x \in \T^{2}} b(x).
\end{equation}
\ \par
\noindent
{\bf 1.} \emph{An enlightening case: constant magnetic field modulus.}
Let us first suppose $b$ constant; for readability we assume here that $b(x):=1$.

As noticed in  \cite[Appendix A, p. 373-374]{OG03}, there are only a finite number of direction in $\mathbb{S}^1$ (identifying $\mathbb{S}^1$ with $[0,2\pi[$, we denote them $\alpha_1,...,\alpha_N \in [0,2\pi[$) for which there exists a half-line in $\mathbb{T}^2$ which does not intersect $B(x_0,r_0/8)$. 
Indeed if the slope is irrational, then each corresponding half-line is dense in the torus, and consequently meets $B(x_0,r_0/8)$. If the slope is rational, say $p/q$ with $p \in \Z$, $q \in \N \setminus \{0\}$ and $\mbox{gcd}(p,q)=1$, then these half-lines $L$ are closed periodic lines in $\T^{2}$.
Due to B\'ezout's theorem, the distance between to consecutive lines in $s^{-1}(L)$ is less than $\min(\frac{1}{|p|}, \frac{1}{q})$, and the conclusion follows. \par
\ \par 
We introduce the neighborhoods of $\alpha_{i}$:
\begin{equation*}
\mathcal{V}_i=(\alpha_i- \beta_i/2, \alpha_i + \beta_i/2),
\end{equation*}
as follows. Let $\beta_i>0$ and $\tau \leq T$ small enough such that
\begin{equation*}
\beta_i < \frac{\tau}{4}
\text{ and }
\frac{\tau}{4} <\min_{i\neq j} d(\mathcal{V}_i,\mathcal{V}_j).
\end{equation*}
By a compactness argument, there exists a length $L>0$ such that for any $x \in \mathbb{T}^2$, $\forall a_i \in \mathbb{S}^1\backslash \cup_{i=1}^N \mathcal{V}_i$, any particle starting from $x$ with a direction $a_i$ has to travel at most a distance $L$ to meet $B(x_0,r_0/8)$.

We fix $m$ large enough such that:
\begin{equation*}
T_{m}:=\frac{L}{m}<\tau/4.
\end{equation*}
This is the time ``free'' particles with velocity $m$ take to cover the distance $L$. We observe that for any $\vert v \vert \geq m$, we have $T_{\vert v \vert} :=  \frac{L}{\vert v \vert}\leq T_m$.

Now let $x\in \mathbb{T}^2, v \in \mathbb{R}^2$ with $|v| \geq {m}$. Let us discuss according to the direction of $v$.

\begin{itemize}
\item First case : $\frac{v}{\vert v \vert} \in  \mathbb{S}^1\backslash \cup_{i=1}^N \mathcal{V}_i$.

We denote $(X^{\#},V^{\#})$ the characteristics associated to free transport.

We have, for any $t<T_{\vert v \vert}$, 
\[
\vert X^{\#}(t+T/4,T/4,x,v)- \overline{X}(t+T/4,T/4,x,v)\vert \leq \vert v \vert \frac{T_{\vert v \vert}^2}{2}= \frac{L^2}{2\vert v \vert}  \leq  \frac{L^2}{2m}.
\]
We can impose $m$ large enough such that $\frac{L^2}{2m} < r_0/8$. As a result:
$$\exists t \in (T/4,T/2], \overline{X}(t,0,x,v) \in B(x_0,r_0/4),$$
and \eqref{GVCM} is trivial here since $|\overline{V}(t,0,x,v)|$ is conserved.

\item Second case :  $\frac{v}{\vert v \vert} \in  \cup_{i=1}^N \mathcal{V}_i$, say $\mathcal{V}_j$.

The idea is to simply wait for a time $\tau/4$. Let us consider 
\begin{equation*}
(x',v'):=(\overline{X} ((T+\tau)/4,T/4,x,v),\overline{V} ((T+\tau)/4,T/4,x,v)).
\end{equation*}
We observe that because of the ``rotation'' induced by the magnetic field and due to the choice of $\beta_{i}$, 
\[
\frac{v'}{\vert v' \vert} \in  \mathbb{S}^1\backslash \cup_{i=1}^N \mathcal{V}_i,
\]
and thus we are in the same case as before. 
\end{itemize}
Consequently we have proven that:
$$\exists t \in (T/4,3T/4], \overline{X}(t,0,x,v) \in B(x_0,r_0/4).$$

\ \\
{\bf 2.} \emph{Positive magnetic field modulus.} Here we suppose that $b>0$ on $\T^{2}$.

We are in the case where in Lemma \ref{LemEpaissi}, we can take $K=K_{d}=\T^{2}$ and 
\begin{equation*}
\underline{b} =\inf_{x \in \mathbb{T}^2} b.
\end{equation*}
Keeping the same notations as before, we set $\tau \in (0,T]$ and $\beta_{i}>0$ in order that

\begin{equation*}
\beta_i < {\underline{b}} \frac{\tau}{4} < \min_{i\neq j} d(\mathcal{V}_i,\mathcal{V}_j).
\end{equation*}

 The proof is very similar to the previous one. Indeed, the following estimate still holds:
\begin{equation} \label{56b}
\vert X^{\#}(t,T/4,x,v)- \overline{X}(t,T/4,x,v)\vert \leq  \frac{L^2}{2m} \overline{b}.
\end{equation}

 Let $\tilde{x}\in \mathbb{T}^2, \tilde{v} \in \mathbb{R}^2$. We distinguish as before between two possibilities.
 Using the previous inequality \eqref{56b}, the first case holds identically for $m$ large.
 For the second case just have to check that with this magnetic field, the velocity is rotated by an angle at least equal to $\beta_i$ after some time $t \in (0, \frac{\tau}{4})$. \par
We use the following computation for general $(x,v)$. Denote by $\theta(t)$ the angle (modulo $2\pi$) between $v^{\perp}$ and $\overline{V}(t,0,x,v)$. Taking the scalar product with $\overline{V}(t,0,x,v)$ in:
 \[
 \frac{d \overline{V}(t,0,x,v)}{dt}= b(\overline{X}(t,0,x,v)) \overline{V}(t,0,x,v)^\perp,
 \]
 we obtain that $\vert \overline{V}(t,0,x,v)\vert=\vert v \vert$. Then, taking the scalar product with $v^{\perp}$, we obtain:
\begin{equation} \nonumber
\sin \theta(t) \theta'(t) = b(\overline{X}(t,0,x,v)) \sin \theta(t),
\end{equation}
so that
\begin{equation} \label{Thetaprime}
\theta'(t) = b(\overline{X}(t,0,x,v)),
\end{equation}
(even if $\sin \theta(t)=0$ in which case one considers the scalar product with $v$.)
We deduce that $\theta'(t) \geq \underline{b}$. \par
Thus going back to $(\tilde{x},\tilde{v})$, by the intermediate value theorem and the definition of the neighborhoods ${\mathcal V}_{i}$, there is a time $T_0$ less or equal to $\tau/4$ for which we have:
 \[
 \overline{V}(T_0 + \frac{T}{4},\frac{T}{4},\tilde{x},\tilde{v}) \in   \mathbb{S}^1\backslash \cup_{i=1}^N \mathcal{V}_i,
 \]
and we conclude as previously.

\ \\
{\bf 3.} \emph{Magnetic field modulus satisfying the geometric condition.} Let us consider the general case for $b$, but without the additional force $\mathfrak{F}$. \par
Given $K$ satisfying the geometric condition \eqref{GeometricCondition}, we introduce $d$ and $D$ as in Lemma \ref{LemEpaissi}. Let 
\begin{equation*}
U:=\T^{2} \setminus K_{d},
\end{equation*}
where we recall the notation \eqref{Kr}. We assume here that $\tau \in (0,T]$ and $\beta_{i}$ are such that
\begin{equation*}
\beta_i < \frac{\underline{b}}{2} \, \inf (\frac{\tau}{4},  \frac{\tau d}{32D}) <\min_{i\neq j} d(\mathcal{V}_i,\mathcal{V}_j).
\end{equation*}
We denote by $(X^{\#},V^{\#})$ the characteristics associated to free transport, while $(\overline{X},\overline{V})$ corresponds to those associated to the magnetic field. \par
\ \par
Let $x\in \mathbb{T}^2, v \in \mathbb{R}^2$. We once again distinguish between the two possibilities.
As before the first case is still similar since \eqref{56b} is still valid. We have to give a new argument for the second case. \par
\ \par
We will assume that $m$ is large enough so that $T_{m} < \frac{\tau}{8}$. We distinguish between several sub-cases:

\begin{enumerate}
\renewcommand{\labelenumi}{\alph{enumi}.}
\item Assume that $\overline{X}(t,0,x,v) \in K_{d}$ for some $t$ in a time interval of length at least equal to $\frac{T}{4}$ inside $[\frac{T}{4}, \frac{3T}{4}]$. Then one can apply the positive magnetic modulus case (case 2).

\item Assume more generally that ${\mathcal L}^{1}( \{ t \in [\frac{T}{4}, \frac{3T}{4}], \overline{X}(t,0,x,v) \in K_{d} \}) \geq T/4$. On $U$, one has $b \geq 0$, so the angle of $V(t,0,x,v)$ with $v$ is non decreasing over time. It follows that we can apply \eqref{Thetaprime} to each passage of the particle in $K_{d}$ and we conclude as before.

\item We assume now that the previous cases do not hold. Then $\overline{X}(t,0,x,v)$ remains in $\mathbb{T}^2 \backslash K_{d}$ at least during a time $\frac{T}{4}$ in $(\frac{T}{4},\frac{3T}{4})$. \par
By \eqref{GrandD}, each passage in $\mathbb{T}^2 \backslash K_{d}$ of $X^{\#}(t,0,x,v)$ lasts at most $D/\vert v \vert$. 
Actually, in $U$, the characteristics $\overline{X}$ are not straight lines since they are modified by the magnetic field. Let us prove nevertheless that if $\vert v \vert$ is large enough, then the particle can remain at most during a time $D/\vert v \vert$ in $U$. \par
Let $x \in U$, and $\frac{v}{\vert v \vert} \in \mathbb{S}_1$, let $\sigma \in (\frac{T}{4},\frac{3T}{4})$. By Lemma \ref{LemEpaissi}, there exists $s< \frac{D}{|v|} $ such that $X^{\#}(\sigma + s,\sigma,x,v) \in K$. Now we can evaluate as for a previous computation:
\begin{equation*}
\left| X^{\#}(\sigma+ s ,\sigma,x,v) - \overline{X}(\sigma + s ,\sigma,x,v) \right| \leq \overline{b} \frac{D^{2}}{2 |v|}.
\end{equation*}
We can choose $m$ large enough such that for any $\vert v \vert \geq m$, $ \overline{X}(\sigma + s ,\sigma,x,v) \in K_{d}$.
Hence at each passage of $X(t,0,x,v)$ in $\mathbb{T}^2 \backslash K_{d}$ lasts at most during a time $D/\vert v \vert$, which proves the claim. \par
This involves that there are at least $\lfloor \frac{T\vert v \vert }{4D} \rfloor -1$ passages in $U$, and therefore there are also at least $\lfloor \frac{T\vert v \vert }{4D}\rfloor -2$ passages in $K_{d}$. This is larger than $\frac{ T\vert v \vert }{8D}$ for $|v|$ large enough.  \par
\medskip
%
%
Now we denote by $t'$ a time for which $\overline{X}(t',0,x,v) \in K_{d}$, with $\overline{X}(t,0,x,v) \notin K_{d}$ for $t< t'$ and $t$ close to $t'$. Let us show that $\overline{X}(t'+s,0,x,v)$ remains in $K_{d}$ for $s \leq \frac{1}{4} \frac{d}{\vert v \vert}$, if the velocity is large enough. We have for all $s \in [0,\frac{1}{4} \frac{d}{\vert v \vert}]$,
\begin{equation*}
\vert X^{\#}(t'+s,t',x,v)- \overline{X}(t'+s,t',x,v)\vert \leq \overline{b} \vert v \vert \frac{\left(\frac{1}{4} \frac{d}{\vert v \vert}\right)^2}{2}.
\end{equation*}
On the other hand, by \eqref{GrandD}, each passage of $X^{\#}$ in $K_{d/2}$ lasts at least $\frac{d}{4\vert v \vert}$.  Hence we can choose $m$ large enough such that for any $\vert v \vert \geq m$, $\overline{X}(t'+s,t',x,v) \in K_{d}$ for $s \in [0,\frac{1}{4} \frac{d}{\vert v \vert}]$.

Consequently, $X(t,0,x,v)$ remains in $K_{d}$ during a time $\frac{T d}{32 D}$ inside $(\frac{T}{4},\frac{3T}{4})$, and we conclude as before.
\end{enumerate}
\ \\
{\bf 4.} {\it With a nontrivial additional force $\mathfrak{F}$.} 

Let us finally explain how one can take $\mathfrak{F}$ into account. 
First, we consider the equations for $|\overline{V}|$ and $\theta$, where $\theta$ is the angle between $v$ and $\overline{V}(t,0,x,v)$. The following computations are valid for $v$ large so that $|\overline{V}(t,0,x,v)|$ does not vanish and for a time interval where $\theta \in [-\pi/2,\pi/2]$.
\par

%
\begin{itemize}
\item  For what concerns $|V|$, it suffices to take the scalar product with $\overline{V}(t,0,x,v)$ of the equation of $\overline{V}$. We infer
\begin{equation*}
\frac{d}{dt} |\overline{V}(t,0,x,v)|^{2} = 2 \mathfrak{F} \cdot \overline{V}(t,0,x,v),
\end{equation*}
so that
\begin{equation} \label{EvolNorme}
\frac{d}{dt} |\overline{V}(t,0,x,v)| =  \frac{\mathfrak{F} \cdot \overline{V}(t,0,x,v)}{\vert \overline{V}(t,0,x,v) \vert}.
\end{equation}
In particular, for $m$ large enough, one has for all $(x,v) \in \mathbb{T}^2 \times \mathbb{R}^2$ \text{ with } $\vert v \vert \geq m$,
\begin{equation} \label{CompTailev}
\frac{|v|}{2} \leq |\overline{V}(t,0,x,v)| \leq 2 |v|.
\end{equation}
\item For what concerns $\theta$, taking the scalar product of the equation of $\overline{V}$ with $v$ we deduce
\begin{multline*}
\left(\frac{d}{dt} |\overline{V}(t,0,x,v)| \right) |v| \cos \theta(t) - |\overline{V}(t,0,x,v)||v|\theta'(t) \sin \theta(t) \\
= b(\overline{X}(t,0,x,v)) \overline{V}^\perp(t,0,x,v) \cdot v + \mathfrak{F} \cdot v.
\end{multline*}
Hence
\begin{multline*}
|\overline{V}(t,0,x,v)||v|\theta'(t) \sin \theta(t) \\ 
= b(\overline{X}(t,0,x,v)) |\overline{V}(t,0,x,v)| |v| \sin (\theta(t)) - \mathfrak{F} \cdot \left(v - \frac{\overline{V}(t,0,x,v) |v|}{|\overline{V}(t,0,x,v)|}\cos \theta(t)\right).
\end{multline*}
We notice that
\begin{equation*}
v - \frac{\overline{V}(t,0,x,v) |v|}{|\overline{V}(t,0,x,v)|}\cos \theta(t) = v - \frac{\overline{V}(t,0,x,v) \cdot v}{|\overline{V}(t,0,x,v)|^{2}} \overline{V}(t,0,x,v) = \mbox{p}_{\{ \overline{V}(t,0,x,v)\}^{\perp}} (v),
\end{equation*}
where $\mbox{p}_{\{ \overline{V}(t,0,x,v) \}^{\perp}} (v)$ denotes the orthogonal projection of $v$ on $\{ \overline{V}(t,0,x,v) \}^{\perp}$.
So
\begin{equation} \label{EvolAngle}
\theta'(t)  = b(\overline{X}(t,0,x,v))  + \frac{1}{|\overline{V}(t,0,x,v)|} \mathfrak{F} \cdot \frac{\mbox{p}_{\{ \overline{V}(t,0,x,v)\}^{\perp}} (v) }{|v|\sin \theta(t)}.
\end{equation}
Note that 
\begin{equation*}
| \mbox{p}_{\{ \overline{V}(t,0,x,v)\}^{\perp}} (v)| = | v | \, |\sin(\theta(t))|,
\end{equation*}
so that:
\begin{equation*}
\frac{1}{|\overline{V}(t,0,x,v)|} \left|\mathfrak{F} \cdot \frac{\mbox{p}_{\{ \overline{V}(t,0,x,v)\}^{\perp}} (v) }{|v|\sin \theta(t)}\right| \geq -\frac{1}{|\overline{V}(t,0,x,v)|} \Vert \mathfrak{F} \Vert_{\infty}.
\end{equation*}
\end{itemize}
\ \par
Now let us revisit the three sub-cases of Case 3 to include $\mathfrak{F}$.
\begin{enumerate}
\renewcommand{\labelenumi}{\alph{enumi}.}
\item Assume that $\overline{X}(t,0,x,v) \in K_{d}$ for all $t$ in a time interval of length at least equal to $\frac{\underline{b} T}{4}$. Then using \eqref{CompTailev} and \eqref{EvolAngle} we deduce
\begin{equation} \label{thetaprime}
\theta'(t) \geq \underline{b} - 2\frac{\| \mathfrak{F} \|_{\infty}}{m},
\end{equation}
so one can conclude as in the positive magnetic modulus case. \par
\item Assume more generally that ${\mathcal L}^{1}( \{ t \in [\frac{T}{4}, \frac{3T}{4}], \overline{X}(t,0,x,v) \in K_{d} \}) \geq T/4$. On $U$, one has $b \geq 0$, so the angle of $V(t,0,x,v)$ with $v$ satisfies
\begin{equation} \label{thetaprime2}
\theta'(t) \geq -  \frac{2}{m}{\| \mathfrak{F} \|_{\infty}},
\end{equation}
and \eqref{thetaprime} when $\overline{X}(t,0,x,v) \in K_{d}$. In total the variation of $\theta$ is no less than $\frac{\underline{b}T}{4} - \frac{T}{2m} \| \mathfrak{F} \|_{\infty}$, so one can conclude as previously (taking $m$ large enough).
\item We assume now that the previous cases do not hold. Then $X(t,0,x,v)$ remains in $\mathbb{T}^2 \backslash K_{d}$ at least during a time $ \frac{T}{4}$ inside $(\frac{T}{4},\frac{3T}{4})$. Let us compare the characteristics $({X},{V})$ associated to $\mathfrak{F}+b(x)v^{\perp}$ with the characteristics $(\overline{X},\overline{V})$ associated to the magnetic field $b(x)v^{\perp}$ alone. \par
\smallskip
Let $x \in U$, and $\frac{v}{\vert v \vert} \in \mathbb{S}_1$, and let let $\sigma \in (\frac{T}{4},\frac{3T}{4})$.  Using the analysis of case 3, there exists $t'< \frac{D}{|v|} $ such that $\overline{X}(\sigma+ t',\sigma,x,v) \in K_{d}$. Now comparing $(\overline{X},\overline{V})$ and $(X,V)$ and using Gronwall's inequality we deduce
\begin{equation} \label{Gronw}
\left\{ \begin{array}{l}
	|V(\sigma+ t',\sigma,x,v) -\overline{V}(\sigma+ t',\sigma,x,v)| \leq \|\mathfrak{F} \|_{\infty} \exp(\| b\|_{W^{1,\infty}} (1+2|v|) t'), \\
	|X(\sigma+ t',\sigma,x,v) -\overline{X}(\sigma+ t',\sigma,x,v)| \leq t' \|\mathfrak{F} \|_{\infty} \exp(\| b\|_{W^{1,\infty}} (1+2|v|) t').
\end{array} \right.
\end{equation}
Using that $|v|t'$ is of order $1$ and taking $m$ large enough, we see that for any $\vert v \vert \geq m$, $ X(\sigma+ t' ,\sigma,x,v) \in K_{3d/2}$.
Hence each passage of $X(t,0,x,v)$ in $\mathbb{T}^2 \backslash K_{3d/2}$ lasts at most $D/\vert v \vert$. 
We deduce as previously that there are at least $\lfloor \frac{T\vert v \vert }{4D}\rfloor -2$ passages of $X(t,0,x,v)$ in $K_{3d/2}$ during $(\frac{T}{4},\frac{3T}{4})$.
\par
\medskip
Now reasoning as in Case 3, using Gronwall's estimate \eqref{Gronw}, we see that if $X(\sigma,0,x,v) \in K_{3d/2}$, and $m$ is large enough, then $X(\sigma+t',0,x,v)$ remains in $K_{2d}$ for all times $t' < \frac{T d}{64 D}$, and we conclude as before.
%
%
\end{enumerate}
\end{proof}
\ \par
Now let us turn to the case of low velocities. This time we proceed as in the case of bounded force fields and prove that an analogue of Proposition \ref{PropAccelerePartout} holds:
\begin{prop}
\label{PropAccelerePartout-mag}
Let $\tau>0$ and $M>0$. There exists $\tilde{M}>0$, ${\mathcal E} \in C^{\infty}([0,\tau] \times \T^2;\R^{2})$ and  $\varphi \in C^{\infty}([0,\tau] \times \T^2;\R)$ satisfying
\begin{gather}
\label{EetPhi-mag}
{\mathcal E} = -\nabla \varphi \text{ in  } [0,\tau] \times (\T^2 \backslash B(x_{0},r_{0})), \\
\label{Phi2SupportTemps-mag}
\mbox{Supp}({\mathcal E}) \subset (0,\tau) \times \T^2, \\
\label{Phi2Harmonique-mag}
\Delta \varphi =0 \text{ in  } [0,\tau] \times (\T^2 \backslash B(x_{0},r_{0})), 
\end{gather}
such that, for any $\mathfrak{F} \in L^{\infty}(0,T;W^{1,\infty}(\T^{2} \times \R^{2}))$ satisfying $\| \mathfrak{F} \|_{L^{\infty}} \leq 1$,  if $({X},{V})$ are the characteristics corresponding the force
%
$\mathfrak{F} + {\mathcal E} + b(x)v^{\perp}$,
\begin{equation} 
\label{AccelerePartout-Mag}
\forall (x,v) \in \T^{2} \times B(0,M), \ 
{V}(\tau,0,x,v) \in B(0,\tilde{M}) \setminus B(0,M+1).
\end{equation}
\end{prop}

\begin{proof}[Proof of Proposition \ref{PropAccelerePartout-mag}]
Again, we introduce $\theta$ and ${\mathcal E}$ as in the proof of Proposition \ref{PropAccelerePartout}. Again, one can choose ${\mathcal C}$ and then $\tau'$ such that
\begin{equation*}
\forall (x,v) \in \T^{2} \times B(0,M), \  \overline{V}(\tau,0,x,v) \in \R^{2} \setminus B(0,M+2 + \tau \| \mathfrak{F} \|_{\infty}).
\end{equation*}
%
%
%
Let us denote by $(\overline{X},\overline{V})$ the characteristics corresponding to the force ${\mathcal E}$ alone. \par
We first observe that we have:
\begin{equation*}
\frac{d}{dt} \vert {V} \vert^2 = (\mathfrak{F}(s,{X},{V}) + {\mathcal E}(s,{X})) \cdot  {V}.
\end{equation*}
Thus, using Cauchy-Schwarz and Gronwall's estimates, we obtain:
\begin{equation*}
\vert  {V} \vert^2 \leq 
\max \left(1,  \vert  v \vert^2 e^{t (\| \mathfrak{F} \|_{\infty} + \| {\mathcal E} \|_{\infty})} \right).
\end{equation*}
We evaluate:
\begin{equation*}
| {X}(t,0,x,v) - \overline{X}(t,0,x,v) |
\leq \int_0^t | {V}(s,0,x,v) - \overline{V}(s,0,x,v) | \, ds,
\end{equation*}
\begin{align*}
| {V}(t,0,x,v) - \overline{V}(t,0,x,v) |
&\leq \int_0^t \Big[ | {\mathcal E} (s, {X}(s,0,x,v)) -{\mathcal E} (s,  \overline{X}(s,0,x,v)) | \\
& \hskip 3cm  + |\mathfrak{F}(t,{X},{V})| + {b} | {V}(s,0,x,v)^\perp |   \Big]ds \\
& \leq  \int_0^t \| \nabla {\mathcal E} \|_{\infty} (t-s) | {V}(s,0,x,v) - \overline{V}(s,0,x,v) | \, ds \\
& \hskip 2cm + \max\left(T,\frac{2M}{\| {\mathcal E} \|_{\infty} + \| \mathfrak{F} \|_{\infty}} (e^{\frac{t}{2}(\| {\mathcal E} \|_{\infty} + \| \mathfrak{F} \|_{\infty})}  -1)\right).
\end{align*}
By Gronwall's inequality:
\begin{equation}
| {V}(t,0,x,v) - \overline{V}(t,0,x,v) | \leq  \max\left(T/2,\frac{2M}{\| {\mathcal E} \|_{\infty} + \| \mathfrak{F} \|_{\infty}} (e^{\frac{t}{2}(\| {\mathcal E} \|_{\infty} + \| \mathfrak{F} \|_{\infty})}-1)\right)   e^{\frac{t^2}{2} \| \nabla^2 \varphi \|_{\infty}}.
\end{equation}

For $t= \tau'$, we have: 
\begin{equation*}
\| \nabla {\mathcal E} \|_{\infty} = \frac{C}{\tau'}, \quad  \| {\mathcal E} \|_{\infty} = \frac{C'}{\tau'},
\end{equation*}
where $C$ and $C'$ depend only on $\omega, M$, and the conclusion follows as previously since
\begin{equation*}
\big|\, |V(\tau,0,x,v)| - |V(\tau',0,x,v)|\,\big| \leq |\tau - \tau'| \| \mathfrak{F} \|_{\infty}.
\end{equation*}
\end{proof}
%
%
%
\noindent
{\bf The reference solution.} Let us now describe the reference solution. Consider $x_0$ in $\omega$ and $r_0>0$ such that $B(x_0,2r_0) \subset \omega$. 
We define the reference potential $\overline{\phi}:[0,T] \times \T^{2} \rightarrow \R$ as follows. We apply Proposition \ref{PropSolRefCM} with $\tau =T/3$, we obtain some $\underline{m}>0$ such that \eqref{GVCM} is satisfied. Then we apply Proposition \ref{PropAccelerePartout-mag} with $\tau =T/3$ and
\begin{equation} \label{DefMHorrible}
M= \max \Big( \underline{m} + \frac{T}{3}, 100, \frac{800 r_{0}}{T}, 32 r_{0} (\overline{b}+1) \Big),
\end{equation}
and obtain some $\overline{\phi}_{2}$, $\overline{{\mathcal E}}_{2}$ and some $\tilde{M}>0$ such that \eqref{AccelerePartout-Mag} is satisfied. We set
\begin{equation*}
\overline{\phi}(t,\cdot) = \left\{ \begin{array}{l}
 0 \text{ for } t \in [0,\frac{T}{3}] \cup [\frac{2T}{3},T], \\
 \overline{\phi}_{2}(t - \frac{T}{3},\cdot) \text{ for } t \in [\frac{T}{3},\frac{2T}{3}],
\end{array} \right.
\end{equation*}
and
\begin{equation*}
\overline{{\mathcal E}}(t,\cdot) = \left\{ \begin{array}{l}
 0 \text{ for } t \in [0,\frac{T}{3}] \cup [\frac{2T}{3},T], \\
 \overline{{\mathcal E}}_{2}(t - \frac{T}{3},\cdot) \text{ for } t \in [\frac{T}{3},\frac{2T}{3}].
\end{array} \right.
\end{equation*}
Then once defined $\overline{\varphi}$, we define $\overline{f}:[0,T] \times \T^{2} \times \R^{2}$ as previously by \eqref{DefZ}-\eqref{Deffbar}.
\subsection{Proof of Theorem \ref{Theo:Mag}} 
We consider ${\mathcal S}_{\varepsilon}$ the same convex set as in the proof of Theorem \ref{Theo:Bounded}, and ${\mathcal V}$ the same fixed point operator with $F=b(x)v^{\perp}$. As before, the proof consists in proving first the existence of a fixed point, and in a second time in proving that such a fixed point is relevant. \par
For what concerns the existence of a fixed point we have:
\begin{lem} \label{Lem:ExistFP2}
There exists $\epsilon_0>0$ such that for any $0<\epsilon<\epsilon_0$, there exists a fixed point of ${\mathcal V}$ in $\mathcal{S}_\epsilon$.
\end{lem}
\begin{proof}[Proof of Lemma \ref{Lem:ExistFP2}]
The proof of Lemma \ref{Lem:ExistFP2} is exactly the same as the one of Lemma \ref{Lem:ExistFP} and is therefore omited. Note in particular that a variant of the crucial Lemma \ref{LemCrucial} is still valid here, using \eqref{EvolNorme}.
\end{proof}
In the second part of the proof we show that a fixed point is relevant. In this part lies the main difference with Theorem \ref{Theo:Bounded}. This is given by the following lemma.
\begin{lem} \label{LemCaracRencontre2}
There exists $\epsilon_1>0$ such that for any $0<\epsilon<\epsilon_1$, all the characteristics $(X,V)$ associated to $b(x)v^{\perp} + \overline{{\mathcal E}} - \nabla \overline{\phi} + \nabla \phi^{f}$, where $f$ is a fixed point of ${\mathcal V}$ in ${\mathcal S}_{\varepsilon}$, meet $\gamma^{3-}$ for some time in $[\frac{T}{12},\frac{11T}{12}]$.
\end{lem}
\begin{proof}[Proof of Lemma \ref{LemCaracRencontre2}]
\ \par
\noindent
%
%
We begin by noticing that $\nabla \varphi^{f}- \nabla \overline{\phi}$ satisfies 
\begin{equation} \label{croix}
\| \nabla \varphi^{f} - \nabla \overline{\phi} \|_{\infty} \leq 1,
\end{equation}
provided that $\varepsilon$ is small enough, which we suppose from now. Consequently we can apply Propositions \ref{PropSolRefCM} and \ref{PropAccelerePartout-mag} to ${\mathfrak F}:= \nabla \varphi^{f}- \nabla \overline{\phi}$. 

It follows that any $(x,v) \in \T^{2} \times \R^{2}$ is (at least) in one of the following situations:
\begin{itemize}
\item If $|V(\frac{T}{3},0,x,v)| \geq M$, then using \eqref{croix}, we deduce $|v| \geq \underline{m}$. Hence
there exists $\tau \in [\frac{T}{12}, \frac{3T}{12}]$ such that
\begin{equation} \label{Danslaboule2}
{X}(\tau,0,x,v) \in B(x_{0},r_{0}/2),
\end{equation}
and reasoning as for \eqref{CompTailev} we deduce that for all $s \in [0,\frac{T}{3}]$ one has
\begin{equation} \label{GrosseVitesse}
 |{V}(s,0,x,v))| \geq \frac{M}{2},
\end{equation}
where $M$ was defined in \eqref{DefMHorrible}.
\item Or $|V(\frac{T}{3},0,x,v)| < M$, so  $|V(\frac{2T}{3},0,x,v)| \geq M+1$, and there exists $\tau \in [\frac{9T}{12}, \frac{11T}{12}]$ such that \eqref{Danslaboule2} is true and \eqref{GrosseVitesse} is valid for all $s \in [\frac{2T}{3},T]$.
\end{itemize}
Let us consider $(x,v)$ in the first situation, the reasoning being identical for the second situation. As in the proof of Lemma \ref{LemCaracRencontre}, we deduce the existence of some $s >0$ with $s < \frac{4 r_{0}}{|v|} \leq \frac{T}{100}$,
\begin{equation} \label{Triangle}
{X}(\tau,0,x,v) - s {V}(\tau,0,x,v) \in S(x_{0}, \frac{3 r_{0}}{2}) \text{ with } {V}(\tau,0,x,v).\nu \leq - \frac{\sqrt{3}}{2} | {V}(\tau,0,x,v)|.
\end{equation}
Let us show that this involves for $|v|$ large enough the existence of $\tau_{*} \in [\tau,t]$ such that
\begin{equation*}
x_{*}:={X}(\tau,0,x,v) - (\tau_{*}-\tau){V}(\tau,0,x,v) \in S(x_{0}, r_{0}).
\end{equation*}
We have for $\sigma \in [\tau-s,\tau]$:
\begin{gather}
\label{XV1}
\frac{M}{2} \leq |{V}(\sigma,0,x,v)| \leq 2 |v|, \\
\label{XV2}
\left| \frac{{V}(\sigma,0,x,v)}{|{V}(\sigma,0,x,v)|}
- \frac{{V}(\tau,0,x,v)}{|{V}(\tau,0,x,v)|} \right| \leq s \Big[ \overline{b} + \frac{2 \| \nabla \varphi^{f} \|_{\infty}}{M}\Big], \\
\label{XV3}
|{X}(\sigma,0,x,v) -{X}(\tau,0,x,v) +(\tau-\sigma) {V}(\tau,0,x,v) | \leq \frac{s^{2}}{2} (2|v| + \| \nabla \varphi^{f} \|_{\infty}).
\end{gather}
Estimate \eqref{XV2} comes from the identity
\begin{equation*}
\frac{d}{d \sigma} \left( \frac{V(\sigma,0,x,v)}{|V(\sigma,0,x,v)|}\right) = \frac{ \frac{d V}{d \sigma} (\sigma,0,x,v)}{|V(\sigma,0,x,v)|} + \frac{\nabla \varphi^{f}(\sigma,x,v) \cdot V(\sigma,0,x,v)}{|V(\sigma,0,x,v)|^{3}} V(\sigma,0,x,v).
\end{equation*}
Let us check that this involves the existence of $t \in [\tau,\tau-s]$ such that $(X(t,0,x,v),V(t,0,x,v)) \in \gamma^{3-}$. The existence of of $t \in [\tau,\tau-s]$ such that $X(t,0,x,v) \in S(x_{0},r_{0})$ follows from \eqref{XV3} and
\begin{equation*}
\frac{s^{2}}{2} (2|v| + \| \nabla \varphi^{f} \|_{\infty}) \leq \frac{8 r_{0}}{|v|^{2}} (2|v| +1) \leq 8 r_{0} \frac{2M+1}{M^{2}}  \leq \frac{24 r_{0}}{M}  \leq \frac{r_{0}}{4}.
\end{equation*}
At such a $t$, from \eqref{XV1}, we have $|V(t,0,x,v)| \geq 2$ since $M \geq 4$. \par
The fact that at such a moment $t$, one has $V(t,0,x,v) . \nu(X(t,0,x,v)) \leq -\frac{1}{5} |V(t,0,x,v,)|$ comes from
\begin{align*}
\Big| \frac{V(t,0,x,v)}{|V(t,0,x,v)|} \cdot \nu(X(t,0,x,v)) & - \frac{V(\tau,0,x,v)}{|V(\tau,0,x,v)|} \cdot \nu(x_{*})  \Big| \\
&\leq 
\Big| \frac{V(t,0,x,v)}{|V(t,0,x,v)|} - \frac{V(\tau,0,x,v)}{|V(\tau,0,x,v)|} \Big| 
+
\Big| \nu(X(t,0,x,v)) - \nu(x_{*}) \Big| \\
&\leq (\overline{b} + \frac{2}{M}) \frac{4r_{0}}{|v|}
+
\frac{1}{r_{0}} |X(t,0,x,v) - x_{*}| \\
&\leq (\overline{b} + 1) \frac{4r_{0}}{M}
+
\frac{24}{M} \leq \frac{1}{4},
\end{align*}
and from \eqref{Triangle}.
This concludes the proof of Lemma \ref{LemCaracRencontre2}. \par
\end{proof}
Let us finally gather all the pieces to prove Theorem  \ref{Theo:Mag}.
\begin{proof}[Proof of Theorem \ref{Theo:Mag}]
Using Lemma \ref{Lem:ExistFP2}, we deduce the existence of some fixed point $f$ of ${\mathcal V}$ in ${\mathcal S}_{\varepsilon}$. Using Lemma \ref{LemCaracRencontre2} we can again use the definitions \eqref{DefU}, \eqref{DefUpsilon} and \eqref{EqLin3} to deduce that $\mbox{Supp\,}[f(T,\cdot,\cdot)] \subset \omega \times \R^{2}$ and one checks that $f$ satisfies the equation for some $G$ as previously. This concludes the proof of Theorem \ref{Theo:Mag}. \par
\end{proof}
\begin{ack}
O. G. is partially supported by the Agence Nationale de la Recherche (ANR-09-BLAN-0213-02). He wishes to thank Institut Henri Poincar\'e (Paris, France) for providing a very  stimulating environment during the ``Control of Partial Differential Equations and Applications'' program in the Fall 2010.
D. H.-K. acknowledges the support of the Agence Nationale de la Recherche (project ANR-08-BLAN-0301-01). He also thanks Colin Guillarmou for some interesting discussions on magnetic fields.
\end{ack}

 \bibliographystyle{plain}
 \bibliography{controlVP}

\begin{thebibliography}{10}

\bibitem{Ar}
A.A. {Arsenev}.
\newblock {Existence in the large of a weak solution of Vlasov's system of
  equations}.
\newblock {\em Z. Vychisl. Mat. Mat. Fiz}, 15:136--147, 1975.

\bibitem{BLR}
C.~Bardos, G.~Lebeau, and J.~Rauch.
\newblock Sharp sufficient conditions for the observation, control, and
  stabilization of waves from the boundary.
\newblock {\em SIAM J. Control Optim.}, 30(5):1024--1065, 1992.

\bibitem{BR}
J.~Batt and G.~Rein.
\newblock Global classical solutions of the periodic {V}lasov-{P}oisson system
  in three dimensions.
\newblock {\em C. R. Acad. Sci. Paris S\'er. I Math.}, 313(6):411--416, 1991.

\bibitem{C92}
J.-M. Coron.
\newblock Global asymptotic stabilization for controllable systems without
  drift.
\newblock {\em Math. Control Signals Systems}, 5(3):295--312, 1992.

\bibitem{C96}
J.-M. Coron.
\newblock On the controllability of {$2$}-{D} incompressible perfect fluids.
\newblock {\em J. Math. Pures Appl. (9)}, 75(2):155--188, 1996.

\bibitem{C07}
J.-M. Coron.
\newblock {\em Control and nonlinearity}, volume 136 of {\em Mathematical
  Surveys and Monographs}.
\newblock American Mathematical Society, Providence, RI, 2007.

\bibitem{OG03}
O.~Glass.
\newblock On the controllability of the {V}lasov-{P}oisson system.
\newblock {\em J. Differential Equations}, 195(2):332--379, 2003.

\bibitem{LP}
P.-L. Lions and B.~Perthame.
\newblock {Propagation of moments and regularity for the three-dimensional
  Vlasov-Poisson system}.
\newblock {\em {Invent. Math.}}, 105:415--430, 1991.

\bibitem{Pfa}
K.~Pfaffelmoser.
\newblock {Global classical solutions of the Vlasov-Poisson system in three
  dimensions for general initial data}.
\newblock {\em {J. Diff. Equations. }}, 95:281--303, 1992.

\bibitem{UO}
S.~Ukai and T.~Okabe.
\newblock On classical solutions in the large in time of two-dimensional
  {V}lasov's equation.
\newblock {\em Osaka J. Math.}, 15(2):245--261, 1978.

\end{thebibliography}
 
\end{document}